\documentclass[10pt,a4paper,twoside,reqno]{amsart}
\author[C.-J. Yao]{Chengjian Yao}
\address{ Institute of Mathematical Sciences, ShanghaiTech University, 393 Middle Huaxia Road, Pudong, 
	Shanghai, 201210 China.}
\email{yaochj@shanghaitech.edu.cn}

\usepackage{hyperref,url, bm}
\usepackage{amssymb,latexsym}
\usepackage{amsmath,amsthm}
\usepackage[latin1]{inputenc}
\usepackage{enumerate}
\usepackage{color}
\usepackage{venndiagram}
\usepackage[all]{xy} \CompileMatrices
\SelectTips{cm}{12} 

\usepackage{verbatim} 

\usepackage{wrapfig}
\usepackage{graphicx}
\usepackage{color}
\usepackage{tikz}
\usepackage{pgfplots}

\def\YYint#1#2#3{{\setbox0=\hbox{$#1{#2#3}{\int}$}
		\vcenter{\hbox{$#2#3$}}\kern-.52\wd0}}

\theoremstyle{plain}
\newtheorem{theorem}{Theorem}[section]
\newtheorem{lemma}[theorem]{Lemma}
\newtheorem{corollary}[theorem]{Corollary}
\newtheorem{proposition}[theorem]{Proposition}
\newtheorem{conjecture}[theorem]{Conjecture}
\newtheorem*{theorem*}{Theorem}
\theoremstyle{definition}
\newtheorem{definition}[theorem]{Definition}
\newtheorem{definition-theorem}[theorem]{Definition-Theorem}

\newtheorem*{acknowledgements}{Acknowledgements}
\theoremstyle{remark}
\newtheorem{remark}[theorem]{Remark}

\newtheorem{question}[theorem]{Question}

\numberwithin{equation}{section} 

\newcommand{\tr}{\operatorname{tr}}

\newcommand{\Aut}{\operatorname{Aut}}

\newcommand{\Vol}{\operatorname{Vol}}

\newcommand{\surj}{\to\kern-1.8ex\to}

\title[Limit of Einstein-Bogomol'nyi metrics]{The dissolving limit and large volume limit of Einstein-Bogomol'nyi metrics}
\date{\today}
\begin{document}
\maketitle

\begin{abstract}
	We study the limits of Einstein-Bogomol'nyi metrics on $\mathbf{P}^1$, which is the solution to a dimensional reduction of Einstein-Maxwell-Higgs system in dimension four, in two regimes. In one regime called the ``dissolving limit''  where the volume of the metrics is approaching the admissible lower bound, it exhibits a pattern that all the vortices are dissolving similar to the Bradlow limit in the study of vortices on Riemann surfaces. In another regime called the ``large volume limit'' where the volume of of the metrics is approaching infinity, the magnetic field is concentrating around the zeros of the Higgs field.  In the meantime, the volume-normalized underlying metric is approaching the Euclidean cone metric determined by the Higgs field in the case of stable Higgs field. Moreover, by studying the large volume limit of Yang's solution for a strictly polystable Higgs field, for each natural number $N'$ we recover the Einstein-Bogomol'nyi metrics on $\mathbf{C}$ which is asymptotically cylindrical at exponential rate and with total string number $N'$ firstly discovered by Linet and Yang. 
\end{abstract}

\section{Introduction}
We are interested in a dimensional reduction of four dimensional Einstein-Maxwell-Higgs equations which is static and enjoys a translation invariance in the third spatial dimension \cite{Yang}. This system has interesting physical background, and describes a model of ``cosmic string'' which provides a potential way of explaining the genesis of \emph{large scale structures} in the universe based on topological defects during the rapid process of phase transition in the very early epoch of the universe. 

Mathematically, the problem is characterized by a system coupling gravity and gauge. The coupling indeed reveals some phenomenons which do not show up in pure gravity theory or pure gauge theory. For instance, in a related but parallel setting, there exists static regular Einstein-Yang-Mills solution for which gravity is coupled with an $SU(2)$ connection \cite{BaMc,SWYM}.

A particular self-dual reduction of Einstein-Maxwell-Higgs system, the so-called Einstein-Bogomol'nyi equations, attracts lot of interests, and many fundamental results have been obtained in the literature, for instance \cite{ComGib, Linet, Yang, Yang4, Yang5, HS}. A keen link of this reduced system with K\"ahler geometry is revealed in \cite{AlGaGa, AlGaGa2}, through an infinite dimensional momentum map picture which is a very successful guiding principle in the study of K\"ahler geometry in recent decades. Using this, we prove a general existence theorem of Einstein-Bogomol'nyi metrics on Riemann sphere in \cite{FPY}, extending the results \cite{Linet, Yang, HS} in a satisfactory way. 

Let us detail out how to formulate Einstein-Bogomol'nyi equation in terms of K\"ahler  geometry. Let $\Sigma$ be a Riemann surface (either compact or noncompact), $L$ be a holomorphic line bundle over $\Sigma$ with a given holomorphic section $\bm\phi$. Consider a pair $(\omega, h)$ which consists of a K\"ahler metric $\omega$ on $\Sigma$ and a Hermitian metric $h$ on $L$, which satisfies the following system
\begin{equation}
\label{eqn:GV0}
\left\{
\begin{array}{r}
iF_{h} + \frac{1}{2} (|\bm\phi|_{h}^2 - \tau)\omega
 = 0,\\
S_{\omega} + \alpha ( \Delta_{\omega} + \tau )( |\bm\phi|_{h}^2 - \tau ) 
= 0. 
\end{array}
\right.
\end{equation}
where $\alpha>0$ is the coupling constant, $\tau>0$ is the symmetry breaking scale. We should make clear the convention we are using here. For the Hermitian metric $h$,  $F_h=-\partial\bar\partial \log h$;  for any smooth function $\eta$ on $\Sigma$, $\Delta_\omega \eta= \tr_\omega \left(-2i\partial\bar\partial \eta\right)= -2 g^{z\bar z} \frac{\partial^2 \eta}{\partial z\partial\bar z}$ when $\omega=i g_{z\bar z}dz\wedge d\bar z$ in a holomorphic local coordinate; moreover, $S_\omega = \tr_\omega \text{Ric } \omega$.

In case $\Sigma$ is compact, the existence of smooth solution to the system implies that $\Sigma=\mathbf{P}^1$ and $\alpha\tau N=1$ where $N=\int_\Sigma c_1(L)$. Take $\omega_0=\omega_{FS}=\frac{idz\wedge d\bar z}{\left(1+|z|^2\right)^2}$ to be the standard Fubini-Study metric on $\mathbf{P}^1$ with volume $2\pi$,  and $h_0=h_{FS}^N$ (where $h_{FS}=\frac{1}{1+|z|^2}$) and writing $h=h_0 e^{2f}$. The second equation in the system is written as 
\begin{equation}
\label{eqn:Ricci-form}
\text{Ric }\omega-2\alpha i\partial\bar\partial |\bm\phi|_h^2 - 2\alpha\tau iF_h=0,
\end{equation}
i.e. 
\begin{align*}
-i\partial\bar\partial \log \frac{\omega}{\omega_0} 
+ 2\omega_0
-2\alpha i\partial\bar\partial |\bm\phi|_h^2 
-2\alpha\tau \left( N\omega_0 - 2i\partial\bar\partial f\right)
=0, 
\end{align*}
which implies that 
\begin{align*}
\log \frac{\omega}{\omega_0} + 2\alpha |\bm\phi|_h^2 -4\alpha\tau f 
= 
\log \frac{1}{\lambda}
\end{align*}
for some constant $\lambda>0$. Plugging the expression 
\[
\omega 
= 
\frac{1}{\lambda}e^{4\alpha\tau f -2\alpha|\bm\phi|_h^2}\omega_0
\]
back into the first equation of the system, it is then transformed to the following PDE 

\begin{equation}
\label{eqn:HS-lambda}
\Delta_{\omega_0} f 
 =
\frac{1}{2 \lambda}
(\tau - |\bm\phi|^2 e^{2f} ) e^{4\alpha\tau f
	-2\alpha |\bm\phi|^2 e^{2f}} 
- 
N
\end{equation}
about $f$, which is called the \emph{Einstein-Bogomol'nyi equation} with parameter $\lambda$. 

Conversely, for any $f$ satisfying \eqref{eqn:HS-lambda} with some $\lambda>0$, it is immediate to see the pair $\left(\omega, h\right)= \left(\frac{1}{\lambda}e^{4\alpha\tau f -2\alpha|\bm\phi|_h^2}\omega_0, h_0e^{2f} \right)$ satisfies \eqref{eqn:GV0}. Such a pair is called an \emph{Einstein-Bogomol'nyi metric} \cite{Yang, Al-Ga-Ga-P}, and we will also denote it by $(\omega, h,\bm\phi)$ if the dependence on the particular Higgs field is to be emphasized.

The study of this particular single semi-linear PDE with one extra parameter $\lambda>0$ is initiated by Yang in \cite{Yang} where he obtains existence of smooth solutions for sufficiently small $\lambda$ under the numerical assumption that either ``$n_j<\frac{N}{2}$ for all $j\in \{1,2, \cdots, d\}$" where $\bm\phi$ vanishes at $d$ distinct points $p_1, \cdots, p_d$ with multiplicities $n_1, \cdots, n_d$ respectively. Moreover, in \cite{Yang5}, existence of $S^1$ symmetric solution is established when $\bm\phi$ is assumed to vanishes at the north and south pole of $S^2$ with equal multiplicity.  Along this line, in \cite{Sohn, HS} Han-Sohn and Sohn show that under the numerical assumption``$n_j<\frac{N}{2}$ for all $j$", there exists some $\lambda_c>0$ such that for any $\lambda\in (0,\lambda_c)$ there exists truly multiple solutions to \eqref{eqn:HS-lambda}, and it is conjectured therein that for any $\lambda>\lambda_c$ there exists no solutions to \eqref{eqn:HS-lambda}. From PDE point of view, the assumption on the multiplicities $n_j$ appears to be only ``technical'', even though \cite{Yang4} does show the nonexistence of $S^1$ invariant solution if all zeros the Higgs field coincide at one point. It's shown in \cite{Al-Ga-Ga-P} Einstein-Bogomol'nyi equation recast in the form of \eqref{eqn:GV0} has a moment map interpretation as in the well-studied case of Hermitian-Einstein metrics and K\"ahler-Einstein metrics. Under this new point of view, the seemingly technical assumption on the multiplicities shows up more naturally as a stability condition. In terms of algebro-geometric language, the assumptions on the multiplicities in the above mentioned existence results precisely mean the divisor $[\bm\phi=0]\in S^N\mathbf{P}^1$ is polystable under the standard $PSL(2,\mathbf{C})$ action in the sense of Geometry Invariant Theory \cite{Mum}. Concretely speaking, suppose $[\bm\phi=0]=\sum_{j=1}^d n_j p_j$, the Higgs field $\bm\phi$ is called stable if $n_j<\frac{N}{2}$ for all $j\in \{1,\cdots, d\}$, and it is called strictly polystable if $d=2$ and $p_1=p_2$.

The methods and results of \cite{Al-Ga-Ga-P, FPY} reveals that the more geometric equation \eqref{eqn:GV0} provides a useful way of studying \eqref{eqn:HS-lambda} which is absent purely based on PDE point of view. Concretely speaking, it is obvious from the first equation of \eqref{eqn:GV0} that any Einstein-Bogomol'nyi metric $(\omega, h)$ on a compact surface must satisfy $\Vol_\omega >\frac{4\pi N}{\tau}:=\underline V$, which is referred as the \emph{admissible lower bound} in this paper. It is proved in \cite{FPY} that for any $V\in \left(\underline{V}, +\infty\right)$ there exists a solution $(\omega, h,\bm\phi)$ to \eqref{eqn:GV0} with $\Vol_\omega=V$ provided $\bm\phi$ is polystable under the $PSL(2,\mathbf{C})$ action in the sense of Geometric Invariant Theory. The way of obtaining those solutions is via a continuity method starting from one particular solution obtained by Yang \cite{Yang} and deforming the volume parameter in the admissible interval $(\underline{V}, +\infty)$. Normalizing the volume for the family of metrics arising in the continuity method to be $2\pi$, we are actually solving the family of equations 

\begin{equation}
\label{eq:FPY-continuity}
\begin{array}{l}
iF_{\tilde h} + \frac{1}{2} \frac{V}{2\pi} (|\bm\phi|_{\tilde h}^2 - \tau)\widetilde\omega 
= 0,\\
S_{\widetilde \omega} + \alpha ( \Delta_{\widetilde\omega} + \tau \frac{V}{2\pi} )( |\bm\phi|_{\tilde h}^2 - \tau )  = 0.
\end{array}
\end{equation}
for $\widetilde \omega\in [\omega_{FS}]$, with the varying parameter $V\in \left(\underline{V}, +\infty\right)$. 

Together with the existence result for the strictly polystable case \cite{Yang4}, the picture for the existence problem of Einstein-Bogomol'nyi metrics on $S^2$ is more or less clear. However, the uniqueness for Einstein-Bogomol'nyi metrics with fixed volume and fixed Higgs field has not been widely explored yet. The only known case is in the strictly polystable situation when the metric is assumed to be $S^1$ symmetric, i.e. the case studied by \cite{Yang4}. Considering the existence result and the above continuity method of varying parameter $V$, it is very natural to consider what might happen in case $V$ tends to the boundary of the admissible interval $\left(\underline V, +\infty\right)$. 

The main aim of this article is to study the limiting behavior of Einstein-Bogomol'nyi metrics with fixed Higgs field in two regimes. In term of the volume parameter, one regime is $V\to \underline V$ and the other is $V\to +\infty$, which is referred as \emph{dissolving limit} and \emph{large volume limit} respectively. 

As $V$ tends to the lower bound $\underline V$, we show that the state function $\left|\bm\phi\right|_h^2$ converges to \emph{zero} uniformly while the K\"ahler metric converges to a suitable multiple of the standard Fubini-Study metric $\omega_{FS}$ on $\mathbf{P}^1$ and the curvature of the Hermitian metric also converges to a suitable multiple of $\omega_{FS}$ (modulo holomorphic automorphism of $\mathbf{P}^1$), see Theorem \ref{thm:dissolving}. It shows a pattern that the unevenness of curvature (or ``vortices'') caused by the Higgs field ``disappears'' in the limit, thus the terminology ``dissolving limit'' is acquired. This represents a similar feature to the moduli space of $N$-vortex solution on $\mathbf{P}^1=S^2$ when the volume of the background round metrics tends to $\underline V=\frac{4\pi N}{\tau}$ \cite{BM}. 

In 
As $V$ tends to $+\infty$, we divides the study into two sub cases: stable case and strictly polystable case, since the geometric picture turns to be rather different. In the strictly polystable case where the Einstein-Bogomol'nyi metrics on $\mathbf{P}^1$ obtained by Linet \cite{Linet} and Yang \cite{Yang4} are relatively more explicitly described as solutions to certain ODE, we are able to study the limit in a quite detailed way. In both regime of limit, the Riemannian metric enjoys the nice property of having uniformly bounded curvature and uniform lower bound on the injectivity radius (Theorem \ref{thm:injectivity}). The metric becomes more and more ``round'' as $V\to \underline V$, while becomes longer and longer (while the size of the central equator is kept bounded) as $V\to +\infty$. Based at points on the central equator, the metric converges to a flat cylinder and the Hermitian metric converges to a Hermitian metric with vanishing curvature. However, based at a fixed zero of the Higgs field and exploring the translation symmetry of the ODE, or $\mathbf{C}^*$ symmetry of $\mathbf{P}^1$, an Einstein-Bogomol'nyi metric whose metric is asymptotically cylindrical and Hermitian metric is asymptotically flat can be constructed. This metric firstly appears in \cite{Linet} (c.f. also \cite{Yang}), and the analysis in the current article puts it into the general framework about ``large volume limit'' of Einstein-Bogomol'nyi metrics on compact surface.

In the stable case, we show that the Einstein-Bogomol'nyi metrics coming from the so-called \emph{maximal} solution of \eqref{eqn:HS-lambda} (partially obtained in \cite{Yang, HS}) exhibits an interesting behavior in the large volume limit. When normalized to a fixed constant volume, the curvature of the Riemannian metric blows up around the zeros of the Higgs field, and the Riemannian metric is proved to converge in Gromov-Hausdorff sense to a flat conical metric $\omega_{(\bm\phi)}=\left|\bm\phi\right|_{h_0}^{-\frac{4}{N}} \omega_0$ determined by $\bm\phi$ on $\mathbf{P}^1$, which has a cone angle $2\pi\beta_j=2\pi\left( 1-\frac{2n_j}{N}\right)$ at $p_j$ for each $j$ (the assumption $2n_j<N$ is also necessary for the existence of such metric on $S^2$). The curvature of the Hermitian metric converges to the Dirac delta current $2\pi \sum_{j=1}^d n_j[p_j]$, and $|\bm\phi|_h^2$ converges to the constant $\tau$ in $C^\infty_{loc}$ sense away from $\{p_1, \cdots, p_d\}$. 

As Euclidean cone metrics show up quite naturally in the large volume limit of Einstein-Bogomol'nyi metrics, the moduli space of Euclidean cone metrics with designated apex curvature (firstly studied by Thurston \cite{Th}) could possibly be viewed as an adiabatic limit of the moduli space of Einstein-Bogomol'nyi metrics with finite volume as the volume tends to $+\infty$, see the discussion in the final section. We also make one conjecture about the large volume limit of Einstein-Bogomol'nyi metrics when the Higgs fields are varying. We leave the study of moduli spaces, including the uniqueness of Einstein-Bogomol'nyi metric with fixed admissible volume and fixed Higgs field, for future investigations.

Geometrically speaking, in the presence of coupling between the spacetime and matter field, in the large volume regime, the flux density/magnetic field strength concentrates more and more along the string locations (corresponding to the zeros of the Higgs field), and the spacetime curvature also concentrates around these strings. The concentration is approximately of cylindrical type in the strictly polystable case and of conical type in the stable case.  

In the literature, several authors have studied large area limit/adiabatic limit of solutions to Abelian vortex equations when the volume of the underlying Riemann surface grows to infinity, \cite{Doan, HJS}. This is named as London limit in the mathematical physics literature about superconducting \cite{Sy}. The convergence result about the Einstein-Bogomol'nyi metrics shows that while the curvature and holomorphic section exhibit similar pattern of convergence, the presence of coupling indeed has back-reaction on the underlying gravity.

\section{Volume and Temper}
In suitable places, we use EB metric for the abbreviation of Einstein-Bogomol'nyi metric. Fixing the symmetry-breaking scale $\tau>0$ and holomorphic line bundle $L$ on $\mathbf{P}^1$ of degree $N$. Define

\begin{itemize}
	\item The full moduli space:
	\begin{equation*}
	\begin{split}
	\mathfrak{M}_{EB}(L, \tau)
	=
	\{
	(\omega, h, \bm\phi)|
	(\omega,h)\text{ is an EB metric with Higgs field }\bm\phi\in H^0(L)	\},
	\end{split}
	\end{equation*}
	\item The partial moduli space: let $\bm\phi\in H^0(L)$, 
	\begin{equation*}
	\begin{split}
	\label{def:moduli-space}
	\mathfrak{M}_{EB}(L, \tau;\bm\phi)
	:=
	\{(\omega, h)| (\omega,h)\text{ is an EB metric for the given }\bm\phi\}.
	\end{split}
	\end{equation*}
\end{itemize}

\noindent Due to the importance of the volume parameter \cite{FPY}, we can refine the above definitions to get the \emph{fixed volume moduli spaces}:

\begin{itemize}
	\item 
	$\mathfrak{M}_{EB}(L, \tau; V)=\{(\omega, h,\bm\phi)\in \mathfrak{M}_{EB}(L, \tau)| \Vol_\omega=V\}$,
	\item 
	$\mathfrak{M}_{EB}(L, \tau; \bm\phi, V)=\{(\omega, h)\in \mathfrak{M}_{EB}(L,\tau; \bm\phi)| \Vol_\omega=V\}$.
\end{itemize}

\begin{definition}[temper]	
	For each Einstein-Bogomol'nyi metric $(\omega, h, \bm\phi)$ on $\mathbf{P}^1$, the corresponding parameter $\lambda$ in \eqref{eqn:HS-lambda} is 
	\begin{equation}
	\lambda 
	= 
	\frac{1}{\Vol_\omega} \int_\Sigma e^{4\alpha\tau f-2\alpha |\bm\phi|_h^2}\omega_0
	=
	\frac{1}{\Vol_\omega}\int_\Sigma |\bm\phi|_h^{4\alpha\tau}|\bm\phi|_{h_0}^{-4\alpha\tau} e^{-2\alpha|\bm\phi|_h^2}\omega_0. 
	\end{equation}
	For simplicity, we give the terminology \emph{temper} to this parameter $\lambda$ and to emphasize the particular dependence we also write $\lambda=\lambda\left(\omega, h, \bm\phi\right)$. 
\end{definition}

It is an interesting question to determine the geometric or physical meaning (not clear yet) of $\lambda$ for an Einstein-Bogomol'nyi metric since it plays important role as an auxiliary parameter in finding the solutions \cite{Yang, Yang4, HS}. To keep more consistent with the literature, we consider $v=2f$, then it satisfies the following nonlinear semi-linear PDE
\begin{equation}
\label{eqn:HS-lambda-v}
\Delta_{\omega_0} v 
=
\frac{1}{\lambda}
(\tau - |\bm\phi|^2 e^v ) e^{2\alpha\tau v
	-2\alpha |\bm\phi|^2 e^v} - 2N. 
\end{equation}
Any solution $v$ gives rise to a solution $\left(\omega, h\right)= \left(\frac{1}{\lambda}e^{2\alpha\tau v -2\alpha|\bm\phi|_h^2}\omega_0, h_0e^v \right)$
to \eqref{eqn:GV0} with respect to the Higgs field $\bm\phi$.

For stable Higgs field $\bm\phi$, the next proposition shows that the temper converges to $0$ if the volume of the Einstein-Bogomol'nyi metrics goes to either end of the admissible interval $\left(\underline V, +\infty\right)$. 
\begin{proposition}
	\label{prop:alternative}
	Let $\bm\phi$ be stable, and $\left(\omega_n, h_n, \bm\phi\right)\in \mathfrak{M}_{EB}\left(L, \tau;\bm\phi\right)$ be a sequence such that one of the two conditions holds
	\begin{enumerate}
		\item[(i)] $\lim_{n\to +\infty} \Vol_{\omega_n}=+\infty$;
		\item[(ii)] $\lim_{n\to +\infty}\Vol_{\omega_n}= \underline V$, 
	\end{enumerate}
	then $\lim_{n\to +\infty}\lambda\left(\omega_n, h_n, \bm\phi\right)=0$. 
\end{proposition}

\begin{proof}
	Argue by contradiction. Suppose there exists $\lambda_0>0$ such that $\lambda_n\geq \lambda_0$. 
	Lemma 2.6 of \cite{HS} implies that (in case $\bm\phi$ is stable) if $v_\lambda$ is a solution to equation \eqref{eqn:HS-lambda-v} for $\lambda\geq\lambda_0$, then for any $\gamma\in(0,1)$ and $p\in (1,+\infty)$ there exists a constant $C$ depending only on $\lambda_0, \gamma, p$ such that 
	\begin{align}
	\label{eqn:bound-potential}
	|\!|v_\lambda|\!|_{C^\gamma} 
	& \leq C, \\
	|\!| v_\lambda|\!|_{W^{2,p}}
	& \leq C. 
	\end{align}
	Then according to the higher order elliptic estimate to the equation \eqref{eqn:HS-lambda-v},  for any $k\in \mathbb{N}^+$ there exists $C_k>0$ such that the $C^k$ norm for $v_\lambda$ is uniformly bounded by $C_k$. 
	
	Applying this estimate to our sequence $v_{\lambda_n}$ coming out of $\left(\omega_n, h_n, \bm\phi\right)$, we can extract a subsequence, still denoted by $v_{\lambda_n}$, such that $v_{\lambda_n}\to v_{\lambda_\infty}$ in $C^2$ sense which solves \eqref{eqn:HS-lambda-v} with parameter $\lambda_\infty>0$ (as $\lambda_n\in [\lambda_0,\lambda_c]$ for any $n$). This implies that $\Vol_{\omega_n}\to \Vol_{g_\infty}$ where $g_\infty = \frac{1}{\lambda_\infty}e^{2\alpha\tau v_{\lambda_\infty} -2\alpha |\bm\phi|^2 e^{v_{\lambda_\infty}}}g_0$ and $h_\infty=h_0 e^{v_{\lambda_\infty}}$ is an Einstein-Bogomol'nyi metric. This is clearly a contradiction with the conditions on the volume behaviors since $\Vol_{g_\infty}\in \left(\underline V, +\infty\right)$. 
\end{proof}

 On the other hand, let $\{v_\lambda\}_{(0,\lambda_1]}$ be a family of solutions to $\eqref{eqn:HS-lambda}_{\lambda}$, then there must hold 
\begin{equation}
\label{fact:C0-blow-up}
\lim_{\lambda\to 0} |\!| v_{\lambda}|\!|_{C^0}=+\infty.
\end{equation}
Argue by contradiction. Suppose there is a sequence $\lambda_i\to 0$ such that  $|\!|v_{\lambda_i}|\!|_{C^0}$ is uniformly bounded, then integrating the first equation in \eqref{eqn:GV0} we obtain
\begin{equation}
\begin{split}
0= \lim_{i\to +\infty} 4\pi N \lambda_i
& = 
\lim_{\lambda\to 0} \int_\Sigma
e^{2\alpha\tau v_{\lambda_i} -2\alpha e^{u_0+v_{\lambda_i}}}  (\tau - e^{u_0+v_{\lambda_i}} )\text{dvol}_{g_0}.\\
\end{split}
\end{equation}
Since the integrand is positive, this implies (by \emph{Riesz's Theorem} in real analysis) there exists a subsequence, still denoted by  $\lambda_i\to 0$ such that 
\[
v_{\lambda_i}\longrightarrow -u_0+\log \tau, \;\;a. e. \text{ on }\Sigma,
\]
contradicting the assumption on the uniform bound of $|\!|v_{\lambda_i}|\!|_{C^0}$.

The estimate \eqref{fact:C0-blow-up} leaves two possibilities for any sequence of solution $v_{\lambda_n}$ with $\lambda_n\to 0$:
\begin{enumerate}
	\item[(I)] $\liminf_{n\to +\infty} \inf_{\mathbf{P}^1} v_{\lambda_n}=-\infty$; 
	\item[(I\!I)] $\limsup_{n\to +\infty} \sup_{\mathbf{P}^1} v_{\lambda_n}=+\infty$.
\end{enumerate}

We will show in this article that the two possibilities both can appear in the case $\bm\phi$ is stable and $\lambda\to 0$. First of all, we show in Proposition \ref{prop:alternative} that one of the \emph{two alternatives} occurs about the volume of the induced Riemannian metric as the temper approaches zero, in case $\bm\phi$ is stable.   

\begin{proposition}
	\label{prop:volume-diverge}
	Suppose $\bm\phi$ is stable. Let  $v_{\lambda_i}$ be a solution to Equation $\eqref{eqn:HS-lambda-v}_{\lambda_i}$ with the sequence $\lambda_i\to 0$ and $g_{\lambda_i}= \frac{1}{\lambda_i}e^{2\alpha \tau v_{\lambda_i} - 2\alpha |\bm\phi|_{h_0}^2 e^{v_{\lambda_i}}}g_0$ be the metric determined by $v_{\lambda_i}$. Then either
	\[
	\lim_{i\to \infty} \Vol_{g_{\lambda_i}}=+\infty,
	\] 
	or 
	\[
	\liminf_{i\to +\infty} \Vol_{ g_{\lambda_i}} = \frac{4\pi N}{\tau}.
	\]
\end{proposition}

\begin{proof}
	We prove this result by contradiction. Assume the volume of $ g_{\lambda_i}$ is uniformly bounded from above. Write $\Phi_{\lambda_i}=|\bm\phi|_{h_{\lambda_i}}^2=e^{u_0+v_{\lambda_i}}$. Define $k_i = e^{2\alpha \Phi_{\lambda_i}} g_{\lambda_i}$ as in \cite{FPY}, then the Lemma 4.9 of \cite{FPY} implies that this sequence of Riemannian metrics on $\mathbf{P}^1$ have uniformly bounded curvature (between $0$ and some uniform constant $K>0$) and covariant derivative. We use the same argument as in \cite[Sect. 5.4]{FPY} to prove that the diameter is uniformly bounded from above.

	Suppose the second alternative in the proposition is not true, then we have a definite distance above the admissible lower bound $\frac{4\pi N}{\tau}$ of the volume. From now on, we are in the same setting as \cite{FPY} for which we could use Cheeger-Gromov's compactness theorem. Thus, we conclude there exists a family $\sigma_i\in SL(2,\mathbb{C})$ and a subsequence of $k_{l_i}$ (still denoted by $k_i$) such that
	\begin{equation}
	\sigma_i^* k_i := k_i' \longrightarrow k_\infty' \text{ in }C^{2,\beta} \text{ sense as } i\to +\infty
	\end{equation}
	for a $C^{2,\beta}$ K\"ahler metric $k'_\infty$ on $\mathbf{P}^1$.  In the meantime, $\sigma_i^{-1}[\bm\phi=0]\rightarrow D_\infty'$. Suppose $D_\infty'$ is not in the $PSL(2,\mathbf{C})$ orbit of $[\bm\phi=0]=D$, then it must consists of either one or two points. Further, since such divisor supports an Einstein-Bogomol'nyi metric it must consist of two points with equal multiplicity by the vanishing of the generalized Futaki's invariant \cite{Al-Ga-Ga-P-Y}. If $D$ is stable, then we get a contradiction since the orbit closure of a stable divisor does not contain a strictly polystable divisor. If $D$ is strictly polystable, then $D_\infty'$ lies in the orbit already since there is only one strictly polystable orbit for this action. In any case, we conclude that $D_\infty'=\sigma(D)$ for some $\sigma\in PSL(2,\mathbf{C})$. 
	
	Now, under the assumption of stability of $D$, the condition $\left(\sigma_i\circ \sigma\right)^{-1}(D)\to D$ means $\sigma_i\circ\sigma$ converges to some element in $Stab(D)\subset SL(2,\mathbf{C})$ which consist of $id, -id$. Since both element acts trivially on $\mathbf{P}^1$, without loss of generality we could assume $\gamma_i:=\sigma_i\circ \sigma\to id$. A consequence of this is that 
	\[
	k_i = \left(\sigma\circ \gamma_i^{-1}\right)^*k_i'\longrightarrow \sigma^*k_\infty' \text{ in }C^{2,\beta} \text{ sense as }i\to +\infty.
	\]

	Take the difference of the two equations
	\[
		\Delta_{k_i} \log \Phi_i = \left(\tau-\Phi_i\right) e^{-2\alpha \Phi_i}  - 4\pi \sum_{j=1}^d n_j \delta_{p_j}
	\]
	and 
	\[
	\Delta_{k_i} \log |\bm\phi|_{h_0}^2 
	= 
	2N \tr_{\omega_{k_i}}\omega_0 - 4\pi \sum_{j=1}^d n_j \delta_{p_j}
	\]
	we get for $v_{\lambda_i}=\log \Phi_i - \log |\bm\phi|_{h_0}^2$, 
	\[
	\Delta_{k_i} v_{\lambda_i} = \left(\tau-\Phi_i\right) e^{-2\alpha \Phi_i}  
	- 
	2N \tr_{\omega_{k_i}}\omega_0. 
	\]
	Since the RHS is uniformly bounded in $L^\infty$ norm, and the coefficient of the Laplacian operator converges in $C^{2,\beta}$ sense, $W^{2,p}$ estimate implies that there exists a uniform constant $C$ such that 
	\begin{equation}
	|\!| v_{\lambda_i}|\!|_{C^0}\leq C. 
	\end{equation}
	This contradicts the previous result \eqref{fact:C0-blow-up}. 	
\end{proof}

The following corollary shows that in case the volume of a sequence of Einstein-Bogomol'nyi metrics goes to $+\infty$, then the $\sup_{\mathbf{P}^1} v_{\lambda_i}$ must diverges to $+\infty$. Actually, 
\begin{proposition}
	\label{cor:factor-converge}
	Let $\bm\phi$ be polystable (either stable or strictly polystable), and $v_{\lambda_i}$ is a family of solutions to Equation $\eqref{eqn:HS-lambda}_{\lambda_i}$ with $\Vol_{g_{\lambda_i}}\to +\infty$. Then there holds 
	\[
	\lim_{i\to \infty} 
	\sup_{\mathbf{P}^1} \left( v_{\lambda_i}+u_0 -\log \tau \right)=0.
	\] 
\end{proposition}

\begin{proof}
	Firstly, we claim 
	\[
	\limsup_{i\to \infty} 
	\sup_{\Sigma} \left( v_{\lambda_i}+u_0 -\log \tau \right)=0. 
	\]
	Suppose the contrary, then according to the fact that $v_\lambda+u_0-\log \tau \leq 0$ for any possible solution $v_\lambda$ to \eqref{eqn:HS-lambda-v}, there exists $\kappa>0$ such that for all $i$
	\[
	v_{\lambda_i} + u_0 - \log \tau \leq -\kappa. 
	\]

	It follows from the scalar curvature formula $S_{g_\lambda}= 2\alpha \frac{|\nabla \Phi_{\lambda}|^2}{\Phi_{\lambda} } + \alpha (\tau - \Phi_{\lambda})^2$ and Gauss-Bonnet Formula that 	
	\[
	4\pi 
	=
	\int_{\mathbf{P}^1} S_{ g_{\lambda_i}} \text{dvol}_{\widehat g_{\lambda_i}}
	\geq
	\alpha 
	\int_{\mathbf{P}^1} (\tau - e^{v_{\lambda_i}+u_0})^2 \text{dvol}_{\widehat g_{\lambda_i}}
	\geq 
	\alpha \tau^2 (1-e^{-\kappa})^2 \Vol_{g_{\lambda_i}},
	\]
	contradicting the assumption $\Vol_{g_{\lambda_i}}\to +\infty$. 
\end{proof}

\begin{remark}
	\label{remark:volume-large}
	We should notice that $\inf_{\Sigma} \left( v_{\lambda_i} + u_0 - \log\tau \right)=-\infty$ always. This corollary does not assume $\bm\phi$ to be stable. However, the Proposition \ref{prop:alternative} shows that in case $\bm\phi$ is stable then the condition ``volume converges to $+\infty$'' implies $\lambda_i\to 0$. In contrast, later in the detailed study of the symmetric solution we will see that it might happen that $\Vol_{g_{\lambda_i}}\to +\infty$ while $\lambda_i$ does not converge to $0$ when $\bm\phi$ is strictly polystable. 
\end{remark}

An immediate corollary of Proposition \ref{prop:volume-diverge} is that

\begin{corollary}
	\label{cor:temper-lower-bound}
	Let $\bm\phi$ be stable. For any $\underline V<V_1<V_2<+\infty$ there exists $\lambda_0=\lambda_0\left(\bm\phi, V_1, V_2\right)>0$ such that  for any $\left( \omega, h, \bm\phi\right)\in \mathfrak{M}_{EB}\left(L, \tau; \bm\phi\right)$, we have  $\lambda\left(\omega, h,\bm\phi\right)\geq \lambda_0$.  
\end{corollary}

Building on the previous results, we can prove the following theorem.

\begin{theorem}
	\label{thm:isolated}
	Let $\bm\phi$ be stable, and $V\in \left(\underline V, +\infty\right)$ be fixed, then 
	$\sharp \big(  \mathfrak{M}_{EB}\left(L, \tau; \bm\phi, V\right)\big)<+\infty$. Moreover, this number is independent of $V$. 
\end{theorem}

\begin{proof}
	For any sequence $\left(\omega_n, h_n, \bm\phi\right)\in \mathfrak{M}_{EB}\left(L, \tau; \bm\phi, V\right)$, the K\"ahler forms $\omega_n$ live inside the fixed de Rham cohomology class $\frac{V}{2\pi}[\omega_0]$. Let $\lambda_n=\lambda\left( \omega_n, h_n, \bm\phi\right)$, and $v_n=\log h_n-\log h_0$ be the corresponding Hermitian potential of $h_n$ (relative to the fixed background Hermitian metric $h_0$).

The function $v_n$ satisfies \eqref{eqn:HS-lambda-v} with parameter $\lambda_n$, i.e. 
	\begin{equation}
	\Delta_{\omega_0} v_n 
	= 
\frac{1}{\lambda_n} \left( \tau - \Phi_0 e^{v_n} \right) e^{2\alpha\tau v_n - \Phi_0 e^{v_n}} - 2N.
	\end{equation}
Since $\Vol_{\omega_n}$ is a constant $V$, the temper $\lambda_n$ for this family of metrics must be uniformly bounded away from $0$ by Corollary \ref{cor:temper-lower-bound}, therefore we have $|\!|v_n |\!|_{C^{k,\gamma}}$ is uniformly bounded for any $k$. This in particular implies that $\lambda_n$ converges to $\lambda_\infty>0$ and $v_n$ converges in $C^\infty$ sense to some $v_\infty$ up to a subsequence. The function $v_\infty$ gives rise to $\left(\omega_\infty, h_\infty, \bm\phi\right)\in \mathfrak{M}_{EB}\left(L,\tau;\bm\phi\right)$ and $\left( \omega_n, h_n\right)$ converges to some  $\left(\omega_\infty, h_\infty\right)$ in $C^\infty$ sense.  
	
The set up of the Implicit Function Theorem argument in \cite[Lemma 3.1]{FPY} implies that for $\left( \omega_\infty, h_\infty\right)$ there exists a sufficiently small neighborhood in suitable Banach space such that any Einstein-Bogomol'nyi metric $(\omega, h)$ in this neighborhood must be uniquely determined by its volume parameter, in other word, there is ``local uniqueness'' result. Therefore, for $n$ sufficiently large, $(\omega_n, h_n)=(\omega_\infty, h_\infty)$. This establishes that the moduli space contains finitely many elements. 	

Fix any volume parameter $V_0\in \left(\underline V, +\infty\right)$, using the continuity path in \cite[Lemma 3.1]{FPY}, each solution $(\omega, h, \bm\phi)\in \mathfrak{M}_{EB}\left( L, \tau;V_0, \bm\phi\right)$ can be generate a smooth family $\Big\{\left(\omega(V), h(V), \bm\phi\right)\Big\}_{V\in \left(\underline V, +\infty\right)}\subset \mathfrak{M}_{EB}\left(L, \tau;\bm\phi\right)$ such that $\Vol_{\omega(V)}=V$ and $\left(\omega(V_0), h(V_0), \bm\phi\right)=\left(\omega, h, \bm\phi\right)$. The smooth curves generated by different elements in $\mathfrak{M}_{EB}\left(L, \tau;V_0, \bm\phi\right)$ do not intersect in $\mathfrak{M}_{EB}\left(L, \tau;\bm\phi\right)$ by the above ``local uniqueness", therefore, $\sharp \big(  \mathfrak{M}_{EB}\left(L, \tau; \bm\phi, V\right)\big)$ is independent of $V$. 
\end{proof}

\begin{remark}
	\label{remark:uniquenss}
According to the general moment map picture of Einstein-Bogomol'nyi metrics with fixed Higgs field and fixed volume \cite{Al-Ga-Ga-P}, there should be ``uniqueness" theorem (modulo holomorphic automorphism of $\mathbf{P}^1$ preserving the Higgs field). We conjecture that for any $V\in \left(\underline V, +\infty\right)$, $\sharp \big(  \mathfrak{M}_{EB}\left(L, \tau; \bm\phi, V\right)\big)=1$ 
for stable $\bm\phi$, and $ \mathfrak{M}_{EB}\left(L, \tau; \bm\phi, V\right)\simeq \mathbf{C}^*$ for strictly polystable $\bm\phi$. 
\end{remark}

\begin{remark}
	Taking Proposition \ref{prop:alternative}, Theorem \ref{thm:isolated} and Remark \ref{remark:uniquenss} into consideration, we have a conjectured dependence of $\lambda$ and $V$ of elements in $\mathfrak{M}_{EB}\left(L, \tau; \bm\phi\right)$ for any stable $\bm\phi$, which is illustrated in the following graph:
	
	\begin{figure}[htb]
		\begin{tikzpicture}[samples=200,domain=0:4]
		\draw[->] (-1.2,0) -- (5,0) node[right] {$volume\;\; V$};
		\draw[dashed] (0.55,0) -- (0.55,2.5) node[right] {};
		\draw[dashed] (-1,2.2) -- (1,2.2) node[right] {};
		\draw[] (-1, 1)  node[left] {} -- (3,1);
		\draw[fill=orange,color=black] (0.55,2.2) circle (1pt);
		\draw[fill=orange,color=black] (-1.1,2.2)  node[left]{critical $\lambda_c$};
		\draw[fill=orange,color=black] (0.55,0)  node[below]{$V_c$};
		\draw[fill=orange,color=blue] (0,0) circle (1pt) node[below]{$\underline V$};
		\draw[fill=orange,color=black] (-1.2,0)  node[below]{$o$};
		\draw[->] (-1,-0.2) -- (-1,3) node[above] {$temper\;\;\lambda$};
		\draw[color=red,<-, domain=0:4.8]  plot (\x,  6*\x/(3*\x^3+1) node[above] {$\lambda$};
		\end{tikzpicture}
		\caption{Conjectured relation of $V$ and $\lambda$ for stable $\bm\phi$}
	\end{figure}
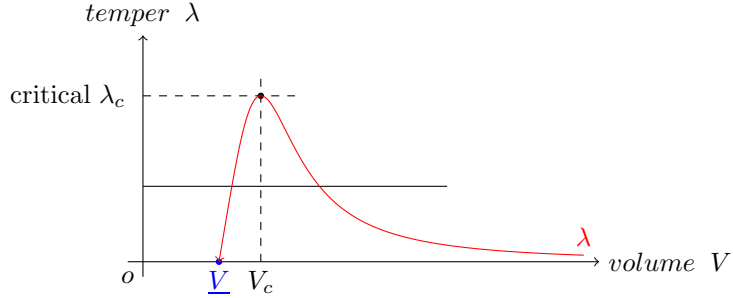

	More precisely, we conjecture that there is a critical $V_c\in \left(\underline V, +\infty\right)$ such that $\lambda$ is strictly increasing on $(\underline V, V_c]$ from $0$ to $\lambda_c$,  and strictly decreasing on $[V_c, +\infty)$ from $\lambda_c$ to $0$. The \emph{maximal solution} of \cite{HS} and the \emph{Leray-Schauder solution} corresponds to volume parameters on the right and left to $V_c$ respectively. 
\end{remark}

\begin{theorem}[Injectivity radius lower bound]
	\label{thm:injectivity}
	Let $\bm\phi$ be polystable (either stable or strictly polystable). For any $\left(\omega, h, \bm\phi\right)\in \mathfrak{M}_{EB}\left(L, \tau\right)$, there holds 
	\[
	inj(\mathbf{P}^1, \omega)
	\geq 
	\frac{\pi}{\sqrt{\frac{\left(3+2\alpha\tau\right)\tau}{2}}}, 
	\]
	and 
	\[
	inj\left( \mathbf{P}^1, e^{2\alpha \Phi}\omega\right)
	\geq 
	\frac{\pi}{\sqrt{\alpha\tau^2}}.
	\]
\end{theorem}
This immediately follows from \cite[Theorem 4.8, Lemma 4.9]{FPY} and the following improvement of the well-known Klingenberg's theorem on the lower bound on injectivity radius. 

\begin{lemma}
	Let $g$ be any smooth Riemannian metric on $S^2$ with $0\leq K_g\leq 1$, then
	\[
	\text{inj}\left(S^2, g\right)\geq \pi. 
	\]
\end{lemma}
If $0<K_g\leq 1$, the stated estimate follows precisely from Klingenberg \cite{Kl}. Under the weaker assumption $0\leq K_g\leq 1$, the Ricci flow initiated from $g$ yield a family of approximation metrics $\{g_t\}_{t\in [0,\epsilon)}$. Since $\int_{S^2} K_g \text{dvol}_g=4\pi>0$, the Gaussian curvature of $g_t$ is positive for $t\in (0,\epsilon)$ (c.f. \cite[Sect 7.1, Chapter 5]{Chow}). For this family, $0<K_{g_t}\leq K_t$ where $K_t=\sup_{S^2} K_{g_t}$ and the rescaled metric satisfies 
\[
0< K_{K_t\cdot g_t}\leq 1. 
\]
Applying the Klingenberg's estimate for $K_t\cdot g_t$ and notice the smooth convergence $K_t\cdot g_t\longrightarrow g$ as $t\to 0$, we obtain that $\text{inj}_g(S^2)\geq \pi$. 

Since this theorem shows a lower bound on injectivity radius independent of the particular choice of Higgs field, it can be used to study the convergence of Einstein-Bogomol'nyi metrics with a sequence of varying Higgs fields.

\section{Dissolving limit}

Take any sequence $(\omega_i, h_i, \bm\phi)\in \mathfrak{M}_{EB}(L, \tau;\bm \phi)$ with $\Vol_{\omega_i}\to \underline V$, integration of the first equation in the system \eqref{eqn:GV0} gives 
\begin{equation*}
\int_{\mathbf{P}^1}\Phi_i \omega_i
=
\tau \Vol_{\omega_i} - 4\pi N,
\end{equation*}
which means that
\begin{equation}
\int_{\mathbf{P}^1}\Phi_i\text{dvol}_{g_i} \to 0. 
\end{equation}
Recall the basic equation about the state function $\Phi_i$:
\begin{equation}
\label{eqn:stateequation}
\Delta_{g_i} \Phi_i
=
- \frac{|\nabla\Phi_i|^2}{\Phi_i} + \Phi_i(\tau-\Phi_i). 
\end{equation}
Integrating on both sides shows that 
\begin{equation}
\int_{\mathbf{P}^1}\frac{|\nabla\Phi_i|^2}{\Phi_i} \text{dvol}_{g_i}
=
\int_{\mathbf{P}^1}\Phi_i(\tau- \Phi_i) \text{dvol}_{g_i}
\leq 
\tau\int_{\mathbf{P}^1}\Phi_i \text{dvol}_{g_i} \to 0, 
\end{equation}
and this fact together with the uniform $C^1$ bound \cite[Corollary 4.6]{FPY}:
\[
|\nabla^{1,0}\bm\phi|_{h_i}^2
=\frac{|\nabla\Phi_i|^2}{\Phi_i}
\leq 
\frac{3\tau}{2\alpha}
\]
implies 
\begin{equation}
\int_{\mathbf{P}^1}\left( \Delta_{g_i} \Phi_i \right)^p \text{dvol}_{g_i} \to 0
\end{equation}
for any $p\geq 1$. 
For the conformally rescaled metric $k_i=e^{2\alpha\Phi_i} g_i$, 
\begin{equation}
\int_{\mathbf{P}^1}(\Delta_{k_i} \Phi_i)^p \text{dvol}_{k_i} 
=
\int_{\mathbf{P}^1}e^{2(1-p)\alpha\Phi_i} \left( \Delta_{g_i} \Phi_i\right)^p \text{dvol}_{g_i} \to 0
\end{equation}
for any $p\geq 1$. 

It is shown in \cite{FPY} that the sequence of metric $k_i=e^{2\alpha\Phi_i} g_i$ has a (subsequential) Cheeger-Gromov limit (since the diameter is uniformly bounded from above) in $C^{2,\beta}$ sense (for any $\beta\in (0,1)$). More precisely, Lemma 4.14 of \cite{FPY} says that for some sequence $\sigma_n\in PSL(2,\mathbb{C})$, $k_n'=\sigma_n^* k_n \to k_\infty'$ in $C^{2,\beta}$ sense. Denote all quantities pulled back under $\sigma_n$ with an extra $'$.  Seeing the fact 
$\int_{\mathbf{P}^1}(\Delta_{k_i'}\Phi_i')^p\text{dvol}_{k_i'} 
=
\int_{\mathbf{P}^1}(\Delta_{k_i} \Phi_i)^p \text{dvol}_{k_i}\to 0$ 
and $\int_{\mathbf{P}^1}\Phi_i'\text{dvol}_{k_i'}=\int_{\mathbf{P}^1}\Phi_i e^{2\alpha\Phi_i} \text{dvol}_{g_i}\to 0$, we could use the standard $W^{2,p}$ estimate (for the family of metrics $k_i'$) to conclude the existence of $K>0$ (independent of $i$) such that
\begin{equation}
|\!| \Phi_i' - \frac{1}{\Vol_{k_i'}}\int_{\mathbf{P}^1}\Phi_i' \text{dvol}_{k_i'} |\!|_{W^{2,p}}^p
\leq 
K \int_{\mathbf{P}^1}(\Delta_{k_i'} \Phi_i')^p \text{dvol}_{k_i'}.
\end{equation}
The family of metrics $k_i'$ obviously has uniform Sobolev constant, and it then follows that as $i\to +\infty$:
\begin{enumerate}
	\item 
	\[
	|\!|\Phi_i|\!|_{C^0}
	= |\!|\Phi_i'|\!|_{C^0}\to 0;
	\]
	\item 
	\[
	|\!|\nabla_{k_i} \Phi_i|\!|_{C^0} 
	=
	|\!|\nabla_{k_i'} \Phi_i'|\!|_{C^0}
	\to 0;
	\]
	\item 
	\[
	|\!| \nabla_{k_i'}\Phi_i'|\!|_{C^\beta}\to 0. 
	\]
\end{enumerate}

When $\bm\phi$ is stable, for any sequence $(\omega_n, h_n, \bm\phi)\in \mathfrak{M}_{EB}(L, \tau;\bm\phi)$ with $\Vol_{\omega_n}\to \underline V$, the above estimates shows that 
	\begin{equation}
	\lambda(\omega_n, h_n, \bm\phi)=
	\frac{1}{\Vol_{\omega_n}}
	\int_{\mathbf{P}^1}|\bm\phi|^{-4\alpha\tau} \Phi_n^{2\alpha\tau}e^{-2\alpha\Phi_n}\omega_0
	\to 0
	\end{equation}
	since $|\bm\phi|^{-4\alpha\tau}$ is $L^1$. This gives another proof to part of Proposition \ref{prop:alternative}.


\vspace{0.5cm}

Let $v_n' = \log \Phi_n' - 4\pi G_n'$ with $G_n'$ being the Green's function of $k_n'$ with an order $n_j$ pole at $p_{j,n}'=\sigma_n^{-1}(p_j)$ for all $j$. It satisfies 
\begin{equation}\label{error-Laplacian-equation}
\Delta_{k_n'}  v_n'
=
(\tau- \Phi_n') e^{-2\alpha \Phi_n'} - \frac{4\pi N}{\nu_n'}.
\end{equation}
where $\nu_n'=\Vol_{k_n'}=\int_{\mathbf{P}^1}e^{2\alpha\Phi_n}\text{dvol}_{g_n}\to \underline V$. The RHS of this equation converges to $0$ in $C^{1,\beta}$ sense according the the above estimates, and therefore the Schauder's estimate implies 

\begin{equation}
|\!| v_n'- \frac{1}{\Vol_{k_n'}}\int_{\mathbf{P}^1}v_n' \text{dvol}_{k_n'} |\!|_{C^{3,\beta}}
\to 
0
\end{equation}
where the $C^{3,\beta}$ is measured in the fixed limit metric $k_\infty'$. In the situation considered in \cite{FPY} where we use continuity method to prove existence of Einstein-Bogomol'nyi metrics with volume $V\in \left( \frac{4\pi N}{\tau}, +\infty\right)$, $\Vol_{\omega_n}$ has a definite distance from $\underline{V}$ and therefore there is some point on $\mathbf{P}^1$ where $\log \Phi_n'=\log\left( \tau -  \frac{4\pi N}{\Vol_{\omega_n}}\right)$ is bounded from below and as a consequence $v_n'$ is uniformly bounded. In the current situation, we are exactly in the case $\Vol_{\omega_n}\to \underline{V}$ and we do not expect $v_n'$ to be uniformly bounded. Indeed, because of the previously established result $\Phi_n'\to 0$, there must hold that
\begin{equation}
\int_{\mathbf{P}^1}v_n' \text{dvol}_{k_n'} 
\to 
-\infty
\end{equation}
and in particular
\begin{equation}
v_n'\longrightarrow_{uniformly}-\infty, 
\end{equation}
and 
\begin{equation}
\label{eqn:shadow}
\frac{\Phi_n'}{e^{\frac{1}{\Vol_{k_n'}}\int_{\mathbf{P}^1}v_n' \text{dvol}_{k_n'}}} \to e^{4\pi G_\infty'}  
\end{equation}
where $G_\infty'$ is the Green function of $k_\infty'$ with poles defined by the limit divisor $D_\infty'=\lim_{n\to \infty} \sigma_n^{-1}(D)$. 

The scalar curvature of the conformal rescaled metric $k_n=e^{2\alpha\Phi_n}g_n$ is
\begin{equation}\label{scalarofauxmetric}
S_{k_n}
=
e^{-2\alpha\Phi_n}(S + \alpha \Delta_{g_n} \Phi_n)
=
\alpha\tau(\tau-\Phi_n)e^{-2\alpha\Phi_n} 
\end{equation}
and converges uniformly to $\alpha\tau^2=\frac{\tau}{N}$, and the same happens for the sequence $k_n'$. Therefore the above limit metric $\omega_{k_\infty'}=\sigma^*\left( \frac{\underline V}{2\pi}\omega_{FS} \right)$ by the uniqueness (module automorphism of $\mathbf{P}^1$) of constant curvature K\"ahler metric inside the K\"ahler class $\frac{\underline V}{2\pi} [\omega_0]$. The consequence is that 
\begin{equation}
\label{eqn:metric-convergence}
(\sigma_n\circ \sigma^{-1})^*\omega_n
\longrightarrow 
\frac{\underline V}{2\pi} \omega_0
\end{equation}
in $C^{1,\beta}$ sense (because $\Phi_n'$ converges in $C^{1,\beta}$ sense to $0$).

 Let $\Sigma$ be a Riemann surface and $L$ be a holomorphic line bundle over $X$. Suppose $\sigma:\Sigma\rightarrow \Sigma$ is a holomorphic automorphism, and $\widetilde \sigma:L\rightarrow L$ is a ``holomorphic lift'' of the action $\sigma$ to the total space, i.e. $\widetilde \sigma$ is a holomorphic map from the total space of $L$ to $L$ covering $\sigma$ and mapping fiber complex linearly to fiber. Let $h$ be a Hermitian metric on $L$, then the formula $\left| p\right|_{\widetilde \sigma^*h}^2:= \left| \widetilde\sigma\left(p\right)\right|_h^2$ (for $p\in L$) defines a new Hermitian metric $\widetilde\sigma^*h$ on $L$, called the pulled-back by $\widetilde\sigma$.  For any smooth section $s$ of $L$, the formula $\left(\widetilde\sigma^*s\right)(y):=\widetilde \sigma^{-1}\left( s\left(\sigma(y)\right)\right)$ defines a smooth section of $L$, called the pulled-back section. The formula 
\[
\left| \left( \widetilde\sigma^* s\right)(y)\right|_{\left(\widetilde\sigma^*h\right)|_y}^2 
= 
\left| \widetilde\sigma\left( \left( \widetilde\sigma^*s\right)(y)\right)\right|_{h|_{\sigma(y)}}^2
=
\left| s\left(\sigma(y)\right)\right|_{h|_{\sigma(y)}}^2
=
\left(  \sigma^*\left( \left|s\right|_h^2\right) \right) (y)
\]
tells that $\left| \widetilde\sigma^*s\right|_{\widetilde\sigma^*h}^2=\sigma^*\left( \left|s\right|_h^2\right)$. It follows that $iF_{\widetilde \sigma^*h}= \sigma^*\left( iF_h\right)$. 

The holomorphic map $\sigma:\mathbf{P}^1\longrightarrow \mathbf{P}^1$ admits a natural linearization, i.e. a lift to a bundle map, $\widetilde \sigma: \mathcal{O}(N)\longrightarrow \mathcal{O}(N)$. For $N=-1$ and $\sigma=\left(\begin{matrix} a & b\\ c & d\end{matrix}\right)\in SL(2,\mathbf{C})$, $\widetilde \sigma$ takes the form 
\[
\left( [z_0:z_1], \left(\zeta_0, \zeta_1\right)\right)\mapsto \left( [a z_0+bz_1: cz_0+dz_1], \left( a\zeta_0+b\zeta_1, c\zeta_0+d\zeta_1\right)\right)
\]
when $\mathcal{O}(-1)\subset \mathbf{P}^1\times \mathbf{C}^2$. Under the above notations, for any $\sigma\in \Aut(\mathbf{P}^1)$ and choice of a holomorphic lift $\widetilde\sigma$, if $\left(\omega, h, \bm\phi\right)\in \mathfrak{M}_{EB}\left(L, \tau;\bm\phi\right)$, then $\left(\sigma^*\omega, \widetilde\sigma^*h, \widetilde\sigma^*\bm\phi\right)\in \mathfrak{M}_{EB}\left(L, \tau; \widetilde\sigma^*\bm\phi\right)$

Relabeling $\sigma_n\circ \sigma^{-1}$ in \eqref{eqn:metric-convergence} as $\sigma_n$, the triple $\left(\omega_n',h_n', \bm\phi_n'\right)=\left( \sigma_n^*\omega_n, \widetilde{\sigma}_n^*h_n, \widetilde\sigma_n^*\bm\phi\right)$  is an Einstein-Bogomol'nyi metric, i.e. the triple satisfies 
\begin{equation}
\label{eqn:GV-pullback}
\left\{
\begin{array}{rl}
iF_{h_n'} + \frac{1}{2} (|\bm\phi_n'|_{h_n'}^2 - \tau)\omega_n'
& = 0,\\
S_{\omega_n'} + \alpha ( \Delta_{\omega_n'} + \tau )( |\bm\phi_n'|_{h_n'}^2 - \tau ) & = 0, 
\end{array}
\right.
\end{equation}
for each $n$.

\begin{theorem}
	\label{thm:dissolving}
	For any sequence of Einstein-Bogomol'nyi metrics $\left(\omega_n, h_n, \bm\phi\right)$ on $\mathbf{P}^1$ with $\Vol_{\omega_n}\to \frac{4\pi N}{\tau}$, there exists a sequence $\sigma_n\in PSL(2,\mathbf{C})$ such that $\sigma_n^*\omega_n$ converges in $C^{1,\beta}$ sense to $\frac{\underline{V}}{2\pi}\omega_0$ and $\widetilde \sigma_n^*h_n$ converges uniformly to $0$ while its curvature converges in $C^{1,\beta}$ sense to $N\omega_0$. 
\end{theorem}

The convergence \eqref{eqn:shadow} suggests that even though $\widetilde\sigma_n^*h_n$ converges to $0$ uniformly, certain kind of ``renormalization'' might be used to see how exactly the Hermitian metric is degenerating to $0$.

\section{Large volume limit}
\label{sect:large}

In this section we will study the limiting behavior of Einstein-Bogomol'nyi metrics when the volume goes to $\infty$ in case of a stable Higgs field $\bm\phi$. The case of a strictly polystable Higgs field will be treated in next section. 

\subsection{Limit of Maximal solution}

Proposition \ref{prop:alternative} says that if $\bm\phi$ is stable, then the temper of a sequence of Einstein-Bogomol'nyi metrics must converge to $0$ if the volume tends to $+\infty$ (notice that this is not the case when $\bm\phi$ is strictly polystable). In the study of solution to \eqref{eqn:HS-lambda-v} \cite{HS}, Han-Sohn found a threshold $\lambda_c>0$ such that for any $\lambda\in (0,\lambda_c)$ there are at least two solutions, one denoted by $v_\lambda^M$ and another by $v_\lambda^{LS}$. The solution $v_\lambda^M$ is found using the supersolution/subsolution method in PDE as \cite{Yang}, called the \emph{maximal solution} or \emph{topological solution} in the literature; the solution $v_\lambda^{LS}$ is obtained via a Leray-Schauder degree theory argument. Presumably these two solutions would behave very differently as $\lambda\to 0$. We study the behavior of the branch of \emph{maximal solution} in this section.

Let  $( g_\lambda^M, h_\lambda^M, \bm\phi)\in \mathfrak{M}_{EB}(L, \tau, \bm\phi)$ be coming from the maximal solution $v_\lambda^M$ to $\eqref{eqn:HS-lambda-v}_\lambda$ constructed by Yang and Han-Sohn \cite{Yang, HS} where $\lambda\in (0,\lambda_c]$. It satisfies the followng estimates and convergence properties( proved in \cite[(1.21)]{HS}) as $\lambda\to 0$:

\begin{align}
\label{est:Han-Sohn-convergence}
& v_\lambda^M 
 \longrightarrow -\log |\bm\phi|_{h_0}^2+\log \tau, \;\;\; a. e. \text{ on }\Sigma;\\
& v_\lambda^M
 \longrightarrow -\log |\bm\phi|_{h_0}^2 + \log \tau \text{ in }W^{1,q}, \;\;\; \forall q\in [1,2);\\
& iF_{ h_\lambda^M} 
=
\frac{1}{2} (\tau - |\bm\phi|_{ h_\lambda^M}^2) \text{dvol}_{ g_\lambda^M}
 \longrightarrow 
2\pi \sum_i n_i [p_i],  \text{ in the sense of measures as } \lambda\to 0;\\
& |\!| v_\lambda^M - (-\log |\bm\phi|_{h_0}^2 + \log \tau) |\!|_{C^s(K)}
 \leq C_{K,s} \lambda\; \text{ for any }K\Subset \mathbf{P}^1\backslash \{\bm\phi=0\}, s\in \mathbb{N}.
\end{align}
Notice that we have not been able to show that the Hermitian potential of the Einstein-Bogomol'nyi metric converges to $-\log |\bm\phi|_{h_0}^2 + \log\tau$ almost everywhere on $\mathbf{P}^1$ as the volume goes to $+\infty$. This is the reason why we consider only the \emph{maximal solution} for which this convergence has been established by \cite{HS}.

The new geometric observations in this article are about the convergence of the underlying Riemannian metrics:
\begin{align}
\lambda  g_\lambda^M
& \longrightarrow 
\frac{\tau^{2\alpha\tau} e^{-2\alpha\tau}}{|\bm\phi|_{h_0}^{4\alpha\tau}} g_0
=
\widehat{g}_{(\bm\phi)}, \text{ in }C^\infty_{loc}(\Sigma\backslash\{p_1,\cdots, p_d\}) \text{ as }\lambda\to 0;\\
\text{Ric }(\lambda  \omega_\lambda^M)
& \longrightarrow 
\frac{4\pi}{N} \sum_i n_i [p_i].
\end{align}
where $\widehat{g}_{(\bm\phi)}$ is obviously the unique \emph{Euclidean cone metric} on $S^2$ with cone-angle $2\pi \beta_j=2\pi (1- \frac{2n_i}{N})$ at $p_i$ (for $i=1,2,\cdots, d$) \cite{Th} and with total volume equals to $V_{\alpha,\tau,\bm\phi}:=\tau^{2\alpha\tau}e^{-2\alpha\tau}\int_{\mathbf{P}^1}|\bm\phi|_{h_0}^{-4\alpha\tau}\omega_0\in (0,+\infty)$, which is abbreviated as $V_{(\bm\phi)}$ for simplicity. The first convergence follows from the last convergence in \eqref{est:Han-Sohn-convergence}, and for the convergence about Ricci form is established using $\text{Ric }(\lambda  \omega_\lambda^M)= 2\alpha i\partial\bar\partial |\bm\phi|_{h_\lambda^M}^2 + 2\alpha\tau iF_{h_\lambda^M}$ and 
\[
\int_{\mathbf{P}^1} \chi i\partial\bar\partial \left|\bm\phi\right|_h^2 
= 
\int_{\mathbf{P}^1} e^{v_\lambda^M}|\bm\phi|_{h_0}^2 \cdot i\partial\bar\partial \chi
\to 
\int_{\mathbf{P}^1}\tau \cdot i\partial\bar\partial \chi
=
0
\]
for any smooth function $\chi$ on $\mathbf{P}^1$ by the Dominate Convergence Theorem. The next proposition says $\Vol_{g_\lambda^M}$ is of the same order of $\frac{V_{(\bm\phi)}}{\lambda}$ as $\lambda\to 0$. 

\begin{proposition}
	If $\lambda_1<\lambda_2$, then $\Vol_{g_{\lambda_1}^M}>\Vol_{g_{\lambda_2}^M}$, and $\lim_{\lambda\to 0} \lambda\cdot\Vol_{g_\lambda^M}= V_{(\bm\phi)}$. 
\end{proposition}

\begin{proof}
	Let us look at the conformal factor of $g_{\lambda}^M$, i.e. 
	\[
	\frac{1}{\lambda}e^{2\alpha\tau v_\lambda^M-2\alpha |\bm \phi|_{h_0}^2 e^{v_\lambda^M}} 
	=
	\frac{1}{\lambda}\mathcal{K}\left(|\bm \phi|_{h_\lambda^M}^2\right)
	\left| \bm \phi\right|_{h_0}^{-4\alpha\tau}
	\]
	with $\mathcal{K}(y)=y^{2\alpha\tau}e^{-2\alpha y}$ defined for $y\geq 0$. 
	
	By the monotonicity formula $v_{\lambda_1}^M>v_{\lambda_2}^M$ proved in \cite{HS} and the fact that $\mathcal{K}$ is strictly increasing on $[0,\tau]$, we have 
	\[
	\frac{1}{\lambda_1}e^{2\alpha\tau v_{\lambda_1}^M -2\alpha |\bm\phi|_{h_0}^2e^{v_{\lambda_1}^M}}
	>
	\frac{1}{\lambda_2}e^{2\alpha\tau v_{\lambda_2}^M -2\alpha |\bm\phi|_{h_0}^2e^{v_{\lambda_2}^M}}, 
	\]
	and therefore $\Vol_{g_{\lambda_1}^M}>\Vol_{g_{\lambda_2}^M}$. As $\lambda\to 0$,  the volume form of $\lambda g_\lambda^M$ pointwisely increases to the volume for of the flat conical metric $\tau^{2\alpha\tau}e^{-2\alpha\tau}|\bm\phi|_{h_0}^{-4\alpha\tau}\omega_0$ and therefore the volume converges to it also.
\end{proof}

\subsection{Gromov-Hausdorff convergence of the rescaled metric}

We have the relation between the rescaled metric $\lambda\omega_\lambda^M$ and other two model K\"ahler metrics on $\mathbf{P}^1$, one is a fixed smooth metric and another is a flat conical metric $\widehat{\omega}_{(\bm\phi)}=\tau^{2\alpha\tau}e^{-2\alpha\tau}|\bm\phi|_{h_0}^{-4\alpha\tau}\omega_0$:
\begin{equation}
\label{eqn:comparison}
e^{2\alpha\tau v_{\lambda_c}^M-2\alpha\tau}\omega_0
\leq 
\lambda\omega_\lambda^M
\leq 
\widehat\omega_{(\bm\phi)}, \;\; \forall \lambda\in (0, \lambda_c].
\end{equation}
The right inequality implies that any two points $q,q'$ on $\mathbf{P}^1$ could be joined by a smooth curve whose length is smaller than any number above the diameter of the metric space where the metric structure is induced from $\widehat\omega_{(\bm\phi)}$. This implies that 
\[
\underline d:=\text{diam}\left( \mathbf{P}^1, e^{2\alpha\tau v_{\lambda_c}^M-2\alpha \tau}\omega_0\right)
\leq \text{diam}\left(\mathbf{P}^1, \lambda \omega_\lambda^M\right)
\leq 
\text{diam}\left(\mathbf{P}^1, \widehat\omega_{(\bm\phi)}\right):=\overline d.
\]
And similarly, the left inequality in \eqref{eqn:comparison} tells that the family has a uniform lower bound on its diameter. This family of Riemannian metrics  $\lambda g_\lambda^M$ has nonnegative curvature and volume upper bound, the volume comparison theorem implies that the volume ratio is uniformly bounded from below, i.e. for any $q\in \mathbf{P}^1$, $\lambda\in (0,\lambda_c]$ and $r\in (0,\overline d]$, 
\begin{equation}
\label{bound:rescaled-injectivity}
\frac{\Vol_{\lambda g_\lambda^M}\left(B_{\lambda g_\lambda^M}(q,r)\right)}{r^2}\geq \kappa, 
\end{equation}
where $\kappa=\frac{\Vol_{e^{2\alpha\tau v_{\lambda_c}^M-2\alpha\tau}\omega_0}}{\overline d^2}$.

Based on the smooth convergence away from the zeros of $\bm\phi$, we can improve the estimate \eqref{est:Han-Sohn-convergence}.
\begin{proposition}~
	\begin{enumerate}
		\item 
	\[
	\lim_{n\to +\infty}\frac{\tau-\max_{\mathbf{P}^1} \Phi_{\lambda_n}^M}{\sqrt{\lambda_n}}=0;
	\]
	   \item 
	   For any $K\Subset \mathbf{P}^1\backslash \{p_1, \cdots, p_d\}$ and $m\in \mathbf{N}$, there exists $C_{K,m}>0$ independent of $n$ such that on $K$
	   \[
	   \left| \nabla_{\omega_0}^m \left( \tau - \Phi_{\lambda_n}^M\right) \right|_{\omega_0}\leq C_{K,m} \lambda_n^2.
	   \]
	\end{enumerate}
\end{proposition}
\begin{proof}
	By the smooth convergence of $\lambda_n \omega_{\lambda_n}^M$ to the flat conical metric away from $\{p_1, \cdots, p_d\}$, the minimum of the curvature of $\lambda_n \omega_{\lambda_n}^M$ with is $\frac{\alpha\left(\tau-\max_{\mathbf{P}^1}\Phi_{\lambda_n}^M\right)^2}{\lambda_n}$ must converge to $0$ as $n\to +\infty$. 
	
	To look at the $C^\infty_{loc}$ convergence of $\lambda_n g_{\lambda_n}^M$, observe that on $K$ by Mean Value Theorem
	\[
	\tau-\Phi_{\lambda_n}^M 
	\leq 
	\tau \left| \log\tau - \left(\log |\bm\phi|_{h_0}^2+v_{\lambda_n}^M \right)\right|
	\leq 
	C_{K,0}\lambda_n
	\]
	by the estimate \eqref{est:Han-Sohn-convergence}. As a consequence, the difference between the volume form of $\tau^{2\alpha\tau}e^{-2\alpha\tau}|\bm\phi|_{h_0}^{-4\alpha\tau}\omega_0$ and $\lambda_n \omega_{\lambda_n}^M$ is controlled by
	\[
	\sup_{p\in K}\max_{y\in \left[ \Phi_{\lambda_n}^M(p), \tau\right]} \mathcal{K}'(y)\left( \tau - \Phi_{\lambda_n}^M(p)\right) |\bm\phi|_{h_0}^{-4\alpha\tau}
	\]
	which is then controlled by $C_{K,0}'\lambda_n$. Similarly, the order $s$ derivative of the the difference is controlled by $C_{K,s}\lambda_n$. 
	
	For part two, use the formula 
	\[
	\frac{\tau - \Phi_\lambda^M}{\lambda}
	= 
	2\tr_{\lambda \omega_\lambda^M} \left( iF_{h_\lambda^M}\right)
	= 
	2\tr_{\lambda \omega_\lambda^M} \left( 
	iF_{h_0} - i\partial\bar\partial v_\lambda^M
	\right)
	= 
	-2i \tr_{\lambda \omega_\lambda^M} \partial\bar\partial \left(v_\lambda^M - \left(-\log|\bm\phi|_{h_0}^2 +\log \tau\right)\right)
	\]
	and the above estimate on $K$, together with induction on the order of differentiation, we conclude there exists $C_{K,m}$ depends only on $K, m$ such that
	\[
	\left|\nabla_{\omega_0}^m \left(\tau - \Phi_{\lambda_n}^M\right)\right|_{\omega_0}\leq C_{K,m}''\lambda_n^2.
	\] 
\end{proof}

\begin{theorem}
As metric spaces, $\left(\mathbf{P}^1, d_{\lambda\cdot g_{\lambda}^M}\right)$ converges to $\left( \mathbf{P}^1, d_{\widehat g_{(\bm\phi)}}\right)$ in Gromov-Hausdorff sense as $\lambda\to 0$. 
\end{theorem}

\begin{proof}
	For any $\varepsilon>0$ small enough, the metric balls  $B_{d_{\widehat{g}_{(\bm\phi)}}}\left(p_j, \varepsilon\right)$ are open disjoint subsets for $j=1,2, \cdots, d$.  Denote $K_\varepsilon= \mathbf{P}^1\backslash \cup_{j=1}^d B_{d_{\widehat{g}_{(\bm\phi)}}}\left(p_j, \varepsilon\right)$. By the relation \eqref{eqn:comparison}, for any $p_j$ and $p\in \partial B_{d_{\widehat g_{(\bm\phi)}}}(p_j,\varepsilon)\subset \partial K_\varepsilon\subset K_\varepsilon$, 
	\[
	d_{\lambda\cdot g_\lambda^M}(p_j,p)\leq 
	d_{\widehat g_{(\bm\phi)}}(p_j,p)=\varepsilon
	\]
	which implies 
	\[
	d_{GH}\left( \left(\mathbf{P}^1, d_{\lambda\cdot g_\lambda^M} \right), \left(K_\varepsilon, d_{\lambda\cdot g_\lambda^M} \right) \right)\leq\varepsilon, \;\; 
	d_{GH}\left( \left(\mathbf{P}^1, d_{\widehat g_{(\bm\phi)}} \right), \left(K_\varepsilon, d_{\widehat g_{(\bm\phi)}} \right) \right)\leq\varepsilon.
	\]
	By the smooth convergence of $\lambda\cdot g_\lambda^M$ to $\widehat g_{(\bm\phi)}$ on $K_\varepsilon$ for fixed $\varepsilon$ as $\lambda\to 0$, for $\lambda$ sufficiently small it holds that
	\[
	d_{GH}\left( \left(K_\varepsilon, d_{\lambda\cdot g_\lambda^M} \right), \left(K_\varepsilon, \widehat g_{(\bm\phi)} \right)\right)
	<\varepsilon. 
	\]
	According to the triangle inequality of Gromov-Hausdorff distance, 
	\[
	d_{GH}\left( \left(\mathbf{P}^1, d_{\lambda\cdot g_\lambda^M} \right), \left(\mathbf{P}^1, d_{\widehat g_{(\bm\phi)}}\right)\right)<3\varepsilon
	\]
	for $\lambda$ sufficiently small. 
\end{proof}

\subsection{Preliminary bubbling analysis}

 The uniform lower bound of volume ratio for the rescaled metric \eqref{bound:rescaled-injectivity} implies the uniform lower bound on the volume ratio for the original family $\omega_\lambda^M$ on any fixed scale. Precisely speaking, take any sequence $\lambda_n\to 0$, and sequence of base point $q_n\in \mathbf{P}^1$, 
\begin{equation}
\frac{\Vol_{g_{\lambda_n}^M}\left(B_{g_{\lambda_n}^M}(q_n, s)  \right)}{s^2} \geq \kappa, \;\; \forall s\in (0,\lambda_n^{-\frac{1}{2}}\overline d].
\end{equation}

We can resort to the idea in \cite{FPY} considering the conformal rescaled sequence $k_n=e^{2\alpha\Phi_{\lambda_n}^M}g_{\lambda_n}^M$ which has uniform bound on the Riemannian curvature and covariant derivative of Riemannian curvature.  Moreover, the above volume ratio lower bound continues to hold (with a possibly smaller $\kappa'>0$). The theorem of Cheeger-Gromov-Taylor shows that $inj(k_n, q_n)$ is uniformly bounded from below. We can therefore take (sub-sequential) pointed Cheeger-Gromov limit: there exists a complete pointed Riemannian manifold $\left(X, k_\infty, q\right)$ such that 
\[
\left(\mathbf{P}^1, k_n, q_n\right)\longrightarrow 
\left(X, k_\infty, q\right), \;\; \text{ in }C^{2,\beta} \text{ Cheeger-Gromov sense}.
\]

Now we take $q_n$ to be one fixed zero $p_j$ of $\bm\phi$, then for any $r<\min_{j\neq i}d_{\widehat g_{(\bm\phi)}}(p_j, p_i)$, the volume of metric balls converges as $\lambda\to 0$, i.e. 
\[
\frac{\Vol_{\lambda \cdot g_\lambda^M}B_{\lambda\cdot g_\lambda^M}(p_j, r)}{\pi r^2}
\to 
\frac{\Vol_{g_{(\bm\phi)}}B_{g_{(\bm\phi)}}(p_j, r)}{\pi r^2}
=
\beta_j = 1-\frac{2n_j}{N}.
\]
For any fixed $s>0$ and $r$ as above, 

\[
\frac{\Vol_{g_\lambda^M}B_{g_\lambda^M}(p_j, s)}{\pi s^2}
=
\frac{\Vol_{\lambda \cdot g_\lambda^M}B_{\lambda\cdot g_\lambda^M}(p_j, \lambda^\frac{1}{2}s)}{\pi \left(\lambda^\frac{1}{2}s\right)^2}
\geq 
\frac{\Vol_{\lambda \cdot g_\lambda^M}B_{\lambda\cdot g_\lambda^M}(p_j, r)}{\pi r^2}
\]
as long as $\lambda^\frac{1}{2}s\leq r$ by the Relative Volume Comparison theorem. Taking limit on both sides as $\lambda\to 0$, we obtain the lower bound on the volume ratio for the limit metric $g_\infty$ on $X$:
\begin{equation}
\frac{\Vol_{g_\infty}B_{g_\infty}(q_j, s)}{\pi s^2}
\geq 
\beta_j
\end{equation}
for any $s>0$. Moreover, 
\[
\int_X S_{k_\infty}dvol_{k_\infty}\leq 4\pi
\]
by the nonnegativity of the Gaussian curvature and Gauss-Bonnet Theorem. Additionally, $S_{k_\infty}=\lim_{n\to +\infty} S_{g_{\lambda_n}^M}(p_j)\geq a\tau^2>0$ and therefore $k_\infty$ is not flat. It follows that $X$ is diffeomorphic to $\mathbf{R}^2$. Spelling out the Cheeger-Gromov convergence above, there exists a decomposition of $X$ into a nested sequence of relatively compact open subsets $X=\cup_{\ell=1}^{+\infty} \Omega_\ell$ with $q\in \Omega_\ell$ and diffeomorphism $F_\ell: \Omega_\ell\longrightarrow \mathbf{P}^1$ with $F_\ell(q)=p_j$, such that $F_\ell^*k_n$ converges to $k_\infty$ as tensors in $C^{2,\beta}$ sense on any compact subset of $X$. Using the equivalence of Hodge star operator and the compatible almost complex structure of the Riemannian metric on two dimensional manifold, we conclude that $F_\ell^*J_{\mathbf{P}^1}$ converges to a ``almost complex structure'' $J_\infty$ on $X$ (in the sense above), and $J_\infty$ has the regularity of $C^{2,\beta}$ and coincides with Hodge star operator of $k_\infty$. By the integrability of $C^{2,\beta}$ almost complex structure on a two dimensional manifold, $(X, J_\infty)$ is an one dimensional complex manifold which is simply connected and the pair $(k_\infty, J_\infty)$ is a K\"ahler structure. 

\begin{theorem}
	Let $\mathbb{D}\subset \mathbf{C}$ be the unit disc, there does not exists complete K\"ahler metric on $\mathbb{D}$ with nonnegative curvature. 
\end{theorem}
\begin{proof}
	Let $\eta = \frac{idz\wedge d\bar z}{\left( 1- |z|^2\right)^2}$ be the standard hyperbolic metric on $\mathbb{D}$, and $\omega$ be any other K\"ahler metric. A direct computation gives the formula 
	\[
	-\frac{1}{2}\Delta_\omega \tr_\omega\eta 
	= 
	2\left( \tr_\omega\eta\right)^2 + S_\omega \cdot \tr_\omega\eta 
	+ 
	\left|\nabla_\omega^{1,0}\eta\right|_{\omega,\eta}^2, 
	\]
	where $\Delta_\omega = \tr_\omega \left(-2i\partial\bar\partial\right)$ is the convention of Laplacian we adopted throughout this article. In the current situation, assume the existence of a K\"ahler metric $\omega$ which is complete and of nonngeative curvature. Since $S_\omega\geq 0$, on $\mathbb{D}$ there holds
	\begin{equation}
	\label{eqn:supercritical}
	-\Delta_\omega \tr_\omega\eta 
	\geq \left(\tr_\omega\eta\right)^2
	\end{equation}
	 Let $d_{p_0}(\cdot)$ be the distance function of the metric $g$ from the point $p_0\in \mathbb{D}$, then 
	\[
	-\Delta_\omega d_{p_0}\leq \frac{1}{d_{p_0}}
	\]
	according to the Laplacian Comparison Theorem for manifold with nonnegative Ricci curvature. Let $w_0>0$ be a real number and $w(x)$ be a real variable function such that 
	\[
	\left\{
	\begin{array}{l}
	w''(x) + \frac{1}{x}w(x) = w(x)^2, \;\; x\in [0,c_{w_0})\\
	w(0)=w_0, w'(0)=0, 
	\end{array}
	\right.
	\]
	where $0<c_{\omega_0}<+\infty$ is such that $[0,c_{w_0})$ is the maximal existence interval for this initial value problem. 
	Define $W(\cdot)=w\left( d_{p_0}(\cdot)\right)$ on $B_{d_{g}}(p_0, c_{w_0})$.  Then $-\Delta_\omega W = w''\left( d_{p_0}\right) + w'(d_{p_0})\cdot \left( - \Delta_\omega d_{p_0}\right)\leq w''(d_{p_0}) + \frac{1}{d_{p_0}}w'(d_{p_0})$, i.e. it satisfies the differential inequality 
	\begin{equation}
	\label{eqn:test}
	-\Delta_\omega W \leq W^2, \;\; \text{on } B_{d_{g}}(p_0, c_{w_0}), 
	\end{equation}
	and $W$ approaches $+\infty$ near $\partial B_{d_{g}}(p_0, c_{w_0})$. The inequality \eqref{eqn:supercritical} and \eqref{eqn:test} implies that 
	\[
	-\Delta_\omega \left( \tr_\omega \eta - W\right) \geq \left( \tr_\omega \eta -W \right)\left( \tr_\omega \eta +W\right)
	\]
	on $B_{d_{g}}(p_0, c_{w_0})$. Since $\tr_\omega \eta - W$ approaches $-\infty$ near the boundary of the metric ball, it must achieves its maximal value at some interior point $p$ of the ball. In particular, $\left( \tr_\omega\eta - W\right)(p)\leq 0$ and therefore 
	\[
	\tr_\omega\eta(p_0)\leq W(p_0)=w_0.
	\]
	Since $w_0>0$ is arbitrarily chosen a priori, we conclude $\tr_\omega\eta(p_0)=0$, which is clearly a contradiction. This contradiction means that a complete K\"ahler metric with nonnegative curvature cannot exists on $\mathbb{D}$. 
\end{proof}

This theorem implies that $\left(X, J_\infty\right)$ cannot be biholomorphic to $\mathbb{D}$, therefore must be biholomorphic to $\mathbf{C}$ by Riemann Mapping Theorem.

The equation \cite[Equation (4.25)]{FPY} about $\Phi_n$ pulls back to $X$ with uniformly bounded right hand side and a family of Laplacian operator whose coefficients converges to that of $k_\infty$ in $C^{2,\beta}$ sense. This implies a $W^{2,p}_{loc}$ and  $C^{1,\gamma}_{loc}$ convergence to a limit function $\Phi_\infty'$ on $X$. Moreover $0\leq \Phi'_\infty\leq \tau$, and satisfies the finite integral condition 
\begin{equation}
\int_X\left(\tau-\Phi_\infty'\right)dvol_{k_\infty} 
\leq 
4\pi N. 
\end{equation}
It also satisfies the equation
\[
\Delta_{k_\infty}\Phi_\infty' = - \frac{\left| d\Phi_\infty'\right|_{k_\infty}^2}{\Phi_\infty'}+ \Phi_\infty'\left(\tau-\Phi_\infty'\right) e^{-2\alpha\Phi_\infty'}. 
\]
One of the remaining difficulties is to show that $\Phi_\infty'$ is positive on $X\backslash\{q\}$ as in \cite{FPY}, i.e. to rule out the case $\Phi_n$ degenerates to $0$ on large region. Our next goal is to show that the asymptotic volume ratio should be equal to $\beta_j$, and this bubbling process is modeled on the family of solution of Chen-Hastings-McLeod-Yang.

\section{Einstein-Bogomol'nyi metrics with symmetry}
The Einstein-Bogomol'nyi equation has been investigated a lot in the mathematical physics literature, and some particular solutions with symmetry were constructed using ODE \cite{Linet, CHMcY, Yang4}. 

\subsection{Limiting behavior of solution with $S^1$ symmetry on $\mathbf{P}^1$}
The main goal of this subsection is to  study the limit behavior of Einstein-Bogomol'nyi metrics on $\mathbf{P}^1$ when the Higgs field is strictly polystable. Very interestingly, in the large volume limit, we recover the asymptotically cylindrical Einstein-Bogomol'nyi metric discovered by Linet \cite{Linet} and Yang \cite{Yang}. 

\subsubsection{The solution of Yang on $\mathbf{P}^1$}
\label{section:regularity}
Choose a strictly polystable divisor $D = \tfrac{N}{2} \cdot 0 + \tfrac{N}{2} \cdot \infty:=N'\cdot 0 + N'\cdot \infty=N'\cdot p_1+N'\cdot p_2$ whose defining section $\bm\phi=z_0^{N'}z_1^{N'}$ is a holomorphic section on $L=\mathcal{O}_{\mathbf{P}^1}(2N')$. Let $u=\log |\bm\phi|_h^2$, then in the cylindrical coordinate $(t,\theta)\in \mathbf{R}\times S^1$ of $\mathbf{C}^*\subset \mathbf{P}^1$ (the transition function between the cylindrical coordinate and the usual polar coordinate $(r,\theta)$ of $\mathbf{C}^*$ is $(r,\theta)=\left(e^t,\theta\right)$, and if needed we will also use the complex coordinate $z=re^{i\theta}$), the equation \eqref{eqn:HS-lambda} (assuming $\tau=1$) about $f$ is translated to 
\begin{equation}
\frac{\partial^2 u}{\partial t^2} + \frac{\partial^2 u}{\partial \theta^2}
= 
\frac{1}{\lambda} e^{2\alpha\left( u-e^u\right)}\left( e^u-1\right)
\end{equation}
with two asymptotic condition $\lim_{t\to +\infty}\frac{\partial u}{\partial t}=-2N'$ and $\lim_{t\to -\infty} \frac{\partial u}{\partial t}=2N'$.

 Yang \cite{Yang4} studied the existence of solutions to this equation assuming the $S^1$-symmetry, i.e. reduced the equation to the ODE initial value problem
\begin{equation}
\label{eqn:ODE}
\left\{
\begin{array}{cc}
u_{tt}
=
\frac{1}{\lambda}e^{2\alpha(u-e^u)}(e^u-1), 
& -\infty<t<+\infty\\
u(0) 
= -\mathfrak{b}, u_t(0)=0
\end{array}
\right.
\end{equation}
satisfying two asymptotic boundary conditions 
\[
\lim_{t\to + \infty} u_t(t)= - 2N', \qquad 
\lim_{t\to -\infty} u_t(t)  = 2N'. 
\]
Using the shooting method of ODE, he showed that for each $\mathfrak{b}>0$ there is a unique parameter
\begin{equation}
\lambda_{\mathfrak{b}} = \frac{1}{2N' e^{2\alpha\left( \mathfrak{b}+ e^{-\mathfrak{b}}\right)}}
\end{equation}
such that the above equation has a global solution $u^{\mathfrak{b}}$ with the prescribed asymptotic boundary conditions on the two ends.\footnote{We should remark that $\pi G$ in Yang's notation is equal to the coupling constant $\alpha$ in our paper, and obviously the symbol $\lambda$ in \cite[Equation (3.1)]{Yang4} is $\frac{1}{\lambda}$ in our notation. Moreover, there is missing coefficient in the exponential shoulder in \cite[Equation (6.7)]{Yang4}}. Indeed, by multiplying both sides of \eqref{eqn:ODE} and integrating from $-\infty$ to $t$, we obtain an explicit formula of $u^\mathfrak{b}_t$ in terms of $u^\mathfrak{b}(t)$ (c.f. also \cite[Equation (6.6)]{Yang4}) 
\begin{equation}
\label{est:gradient}
u^\mathfrak{b}_t(t)^2 
= 
4N'^2 - \frac{1}{\lambda_\mathfrak{b} \alpha} e^{2\alpha\left(u^\mathfrak{b}(t)-e^{u^\mathfrak{b}(t)}\right)}. 
\end{equation}
For any $\delta>0$, there exists $T^\mathfrak{b}_\delta>0$ sufficiently large such that for any $t<-T^\mathfrak{b}_\delta$, 
\[
u^\mathfrak{b}_t(t)>2N'-\delta, \;\;
2N't-\mathfrak{b}\leq u^\mathfrak{b}(t)<\left(2N'-\delta\right)t.
\]
The formula \eqref{est:gradient} reads as
\begin{equation*}
2N'-u^\mathfrak{b}_t(t)
=
\frac{1}{\alpha\lambda_\mathfrak{b} \left( 2N'+u^\mathfrak{b}_t(t)\right)} e^{2\alpha\left( u^\mathfrak{b}(t) -e^{u^\mathfrak{b}(t)}\right)}, 
\end{equation*}
which implies that for $t<-T_\delta$, 
\begin{equation}
\label{est:exponentialdecay}
0
< 2N'-u^\mathfrak{b}_t(t)
\leq 
\frac{1}{\alpha\lambda_{\mathfrak{b}}\left( 4N'-\delta\right)} e^{2\alpha \left(2N'-\delta\right)t}, 
\end{equation}
i.e. $u^\mathfrak{b}_t(t)$ converges to $2N'$ as $t\to -\infty$ at an exponential rate. Similarly, using the fact that $u^\mathfrak{b}_t(t)\to -2N'$ (as $t\to +\infty$) , we can show $u^\mathfrak{b}_t(t)$ converges to $-2N'$ as $t\to +\infty$ at an exponential rate. 

Define $v^\mathfrak{b}(r,\theta)= u^\mathfrak{b}\left(\log r\right) - 2N'\log r$, then it satisfies 
\begin{equation}
\label{eqn:remainder}
\Delta_{g_{euc}} v^\mathfrak{b}
=
\frac{1}{\lambda_\mathfrak{b}} e^{2\alpha \left( v^\mathfrak{b}-r^{2N'}e^{v^\mathfrak{b}}\right)} \left( 1 - r^{2N'}e^{v^\mathfrak{b}}\right) 
\end{equation}
where we use the convention $\Delta_{g_{euc}} v=-\left( v_{rr} + \frac{1}{r}v_r + \frac{1}{r^2} v_{\theta\theta} \right)$. Moreover, using \eqref{est:exponentialdecay} we get 
\[
-\mathfrak{b}\leq 
u^\mathfrak{b}(t) - 2N't 
\leq 
2N'T^\mathfrak{b}_\delta  + u^\mathfrak{b}\left(-T^\mathfrak{b}_\delta\right) 
+ 
\frac{1}{4\alpha^2 \lambda_\mathfrak{b}\left(4N'-\delta\right)\left(2N'-\delta\right)}\left( 
e^{-2\alpha(2N'-\delta)T^\mathfrak{b}_\delta} - e^{2\alpha(2N'-\delta)t}
\right)
\]
for $t<-T^\mathfrak{b}_\delta$. In particular, $u^\mathfrak{b}(t)-2N't$ is bounded on $(-\infty, -T^\mathfrak{b}_\delta]$, and the RHS of \eqref{eqn:remainder} is bounded on the unit disc $\mathbb{D}\subset \mathbf{C}$. In the meantime, 
\begin{equation}
\left| \frac{\partial v^\mathfrak{b}}{\partial r}\right|
=
\left| \frac{u^\mathfrak{b}_t\left(\log r\right) - 2N'}{r}\right|
\leq 
\frac{1}{\alpha\lambda_\mathfrak{b}(4N'-\delta)}r^{2\alpha(2N'-\delta) -1}.
\end{equation}
The above arguments show that $v^\mathfrak{b}\in W^{1,2}\left(\mathbb{D}\right)\cap C^\infty\left(\mathbb{D}^*\right)$ satisfies $\Delta_{g_{euc}} v^\mathfrak{b}\in L^\infty(\mathbb{D})$. The standard elliptic regularity theory shows that $v^\mathfrak{b}\in W^{2,p}\left(\mathbb{D}\right)$ for any $p\geq 1$. This in particular implies that $v^\mathfrak{b}\in C^{1,\beta}\left(\mathbb{D}\right)$ for any $\beta\in (0,1)$. Standard bootstrapping argument to \eqref{eqn:remainder} implies that $v^\mathfrak{b}\in C^\infty\left(\mathbb{D}\right)$. 

Geometrically, we define $h_\mathfrak{b}(z)=e^{v^\mathfrak{b}(z)}$ under the standard trivialization of  the line bundle $\mathcal{O}_{\mathbf{P}^1}(2N')$ over the patch $\mathbf{C}$, i.e. define $\left|\bm\phi\right|_{h_\mathfrak{b}}^2 = \left|z^{N'}\right|^2 e^{v^\mathfrak{b}(z)}$, we obtain a smooth Hermitian metric on the line bundle on $\mathcal{O}_{\mathbf{P}^1}(2N')$, still denoted by $h_\mathfrak{b}$ . Moreover, we define a smooth Riemannian metric
\begin{equation}
\begin{split}
g_{\mathfrak{b}}
& =
\frac{1}{\lambda_{\mathfrak{b}}} e^{2\alpha \left( u^{\mathfrak{b}}- e^{u^{\mathfrak{b}}}\right)} r^{-2} g_{euc}
=
\frac{1}{\lambda_{\mathfrak{b}}} e^{2\alpha \left( u^{\mathfrak{b}}- e^{u^{\mathfrak{b}}}\right)}\left( dt^2+d\theta^2\right)
= 
\frac{1}{\lambda_\mathfrak{b}} e^{2\alpha v^\mathfrak{b}(z)-2\alpha \left|z\right|^{2N'}e^{v^\mathfrak{b}(z)}}g_{euc}
\end{split}
\end{equation}
on $\mathbf{C}\subset \mathbf{P}^1$ which can be verified to extends as a smooth metric on another patch $\mathbf{C}=\mathbf{P}^1\backslash \{0\}$. The pair $(g_{\mathfrak{b}}, h_{\mathfrak{b}})$ is a $S^1$-symmetric solution to the Einstein-Bogomol'nyi equations with $\tau=1$. 

Now we want to look at the limit behavior of the $\left( g_\mathfrak{b},h_\mathfrak{b}\right)$ as $\mathfrak{b}\to 0+$ and $\mathfrak{b}\to +\infty$.

\subsubsection{Limit as $\mathfrak{b}\to +\infty$}
It is evident that $u^\mathfrak{b}$ goes to $-\infty$ uniformly on $(-\infty,+\infty)$, therefore it is useful to look at the ``remainder'' function. In this perspective, let $w^\mathfrak{b}=u^\mathfrak{b}+\mathfrak{b}$, then $w^\mathfrak{b}$ satisfies 
\begin{equation}
\begin{split}
\frac{d^2}{dt^2}w^\mathfrak{b}
& =
2N' e^{2\alpha e^{-\mathfrak{b}}\left(1-e^{w^\mathfrak{b}}\right)+2\alpha w^\mathfrak{b}}\left( e^{-\mathfrak{b}} e^{w^\mathfrak{b}}-1\right);\\
w^\mathfrak{b}(0)
& = \frac{d}{dt}|_{t=0}w^\mathfrak{b}=0. 
\end{split}
\end{equation}
Using the estimates $-2N'|t|\leq w^\mathfrak{b}(t)\leq 0$, $-2N'\leq \frac{d}{dt}w^\mathfrak{b}\leq 2N'$ we obtain $-2N'\leq \frac{d^2}{dt^2}w^\mathfrak{b}\leq 0$ on $\mathbf{R}$. It can then be deduced inductively that there exists $C_i>0$ ($i=2,3,\cdots$) such that $
\left| \frac{d^i}{dt^i}w^\mathfrak{b}\right|\leq C_i
$
on $\mathbf{R}$. By Arzela-Ascoli's Theorem, for any sequence $\mathfrak{b}_n\to +\infty$, there exists a subsequence, still denoted by $\mathfrak{b}_n$ such that $w^{\mathfrak{b}_n}$ converges to $w^{(\infty)}$ in $C^\infty_{loc}$ sense on $\mathbf{R}$. The limit function satisfies 
\begin{equation}
\begin{split}
\frac{d^2}{dt^2}w^{(\infty)}
& =
- 2N' e^{2\alpha w^{(\infty)}};\\
w^{(\infty)}(0)
& = \frac{d}{dt}|_{t=0}w^{(\infty)}=0. 
\end{split}
\end{equation}
Elementary integration gives an explicit formula $w^{(\infty)}(t)= N'\log \frac{1}{\cosh^2 t}$. The curvature of the Hermitian metric $iF_{h_{\mathfrak{b}_n}}=-i\partial\bar\partial u^{\mathfrak{b}_n}=-i\partial\bar\partial w^{\mathfrak{b}_n}$ converges to $-i\partial\bar\partial w^{(\infty)}=-\frac{1}{2}\frac{d^2}{dt^2}w^{(\infty)}dt\wedge d\theta = \frac{N'}{\cosh^2 t}dt\wedge d\theta= 2N'\omega_0$ in the similar sense as above.  Similarly, the Riemannian metric $g_{\mathfrak{b}_n}=2N' e^{2\alpha e^{-{\mathfrak{b}_n}}}(1-e^{w^{\mathfrak{b}_n}}) +2\alpha w^{\mathfrak{b}_n}\left(dt^2+d\theta^2\right)$ converges to $2N' e^{2\alpha w^{(\infty)}}\left(dt^2+d\theta^2\right)=4N'\omega_0$ in $C^\infty_{loc}\left(\mathbf{P}^1\backslash\{p_1,p_2\}\right)$ sense. Since the limit is unique, that means the whole family converges. If we denote $\widetilde h_\mathfrak{b}=h_\mathfrak{b}e^\mathfrak{b}$, then the ``rescaled'' state function $\widetilde \Phi_\mathfrak{b}= \left|\bm\phi\right|_{\widetilde h_\mathfrak{b}}^2= e^{w^\mathfrak{b}}$ converges to the limit $e^{w^{(\infty)}}= \frac{1}{\cosh^{2N'}t} = {\left|z^{N'}\right|^2}/{\left( \frac{1 + \left|z\right|^2}{2}\right)^{2N'}}$.

This phenomenon can be summarized as the following, which explicitly verifies partially Theorem \ref{thm:dissolving} in the strictly polystable case:

``\emph{ As $\mathfrak{b}\to +\infty$ for the family of Einstein-Bogomol'nyi metrics $\left( g_\mathfrak{b}, h_\mathfrak{b}\right)$ with Higgs field $\bm\phi=z_0^{N'}z_1^{N'}$, the metric $g_\mathfrak{b}$ converges to the round metric on $\mathbf{P}^1$, the state function $|\bm\phi|_{h_\mathfrak{b}}^2$ converges to $0$ uniformly and the ``rescaled'' Hermitian metric $\left( \max_{\mathbf{P}^1} \left|\bm\phi\right|_{h_\mathfrak{b}}^2 \right)^{-1}\cdot h_\mathfrak{b}$ converges to the constant curvature metric on the line bundle $\mathcal{O}_{\mathbf{P}^1}(2N')$.}"

\subsubsection{Limit as $\mathfrak{b}\to 0+$}
\label{section:cylindrical}
It is shown in \cite{FPY} as $\mathfrak{b}\to 0^+$,  $\Vol_{g_\mathbf{b}}\to +\infty$. Similar to it, using the estimate $-\mathfrak{b}
\geqslant 
u^\mathfrak{b}(t)
\geqslant -\mathfrak{b} - N |t|\;\; \left( \forall t\in \mathbf{R}\right)$ and the convergence of $u^\mathfrak{b}$, we conclude $\text{diam}\left(\mathbf{P}^1,g_\mathfrak{b}\right)\to +\infty$ since
for any fixed $T>0$, 
\begin{align*}
\liminf_{\mathfrak{b}\to 0+}\text{diam}(\mathbf{P}^1, g_\mathfrak{b})
& \geq 
\liminf_{\mathfrak{b}\to 0+} d_{g_\mathfrak{b}}(p_1,p_2)\\
& =
\liminf_{\mathfrak{b}\to 0+} \frac{2}{\sqrt{\lambda_\mathfrak{b}}} \int_0^T e^{\alpha \left( u^\mathfrak{b}(t)-e^{u^\mathfrak{b}(t)}\right)} dt\\
& =
\frac{2}{\sqrt{\lambda_\mathfrak{b}}}e^{-\alpha}T.
\end{align*}

Because of the unbounded diameter for the sequence of metrics, we need to look at pointed convergence. Let us first choose the base point $q_n=p$ to be a fixed point on the central equation circle.  The length of the particular central equator circle $(t=0)$ is $\frac{1}{\sqrt{\lambda_\mathfrak{b}}}e^{\alpha\left( -\mathfrak{b}-e^{-\mathfrak{b}}\right)}\cdot 2\pi$ which converges to $2\pi e^{-\alpha}\sqrt{Ne}$. Actually, by the uniform $C^k$ convergence of $u^\mathfrak{b}$ to constant function $0$ on any fixed interval $[-T,T]$, the family converges in (pointed) Cheeger-Gromov sense to the flat metric $g_\infty= Ne^{1-2\alpha} \left(dt^2+d\theta^2\right)$ on the cylinder $X=\mathbf{R}\times S^1=\mathbf{C}^*$. Actually, simply taking $F_\mathfrak{b}: X_\mathfrak{b}=\left[-\frac{1}{\mathfrak{b}},\frac{1}{\mathfrak{b}}\right]\times S^1\longrightarrow \mathbf{P}^1=\mathbf{C}\cup\{\infty\}$ sending $\left(t,\theta\right)$ to $\left(r,\theta\right)=\left( e^t,\theta\right)$, we can see clearly $F_\mathfrak{b}^*g_\mathfrak{b}$ converges to the above flat metric $g_\infty$ on any compact subset $\left[-T,T\right]\times S^1\subset X$. It is a very special situation that all the local diffeomorphism $F_\mathfrak{b}$ in the Cheeger-Gromov convergence patches together to one diffeomorphism $F:X\longrightarrow \mathbf{P}^1$. Along the Cheeger-Gromov convergence, $F_\mathfrak{b}^*J_{\mathbf{P}^1}=J_{\mathbf{C}^*}$ converges to $J_{\mathbf{C}^*}$ which makes $g_\infty$ a K\"ahler metric. Since $F_\mathfrak{b}$ is holomorphic, $F^*L$ can be naturally equipped with a holomorphic structure and there exists a holomorphic bundle map $\widetilde F: F^*L\longrightarrow L$ lifting $F$. The Hermitian metric $h_{\mathfrak{b}}= \frac{e^{u^\mathfrak{b}}}{\left| z^{N'} \right|^2}$ is pulled back to $\widetilde F^*h_{\mathfrak{b}}= e^{u^\mathfrak{b}(t)-2N't}$, and converges to $h_\infty = e^{-2N't}$ in smooth sense on any compact subset of $X$. The holomorphic section $\bm\phi$ of $L$ is pulled back to $\widetilde F^*\bm\phi=\left( e^{t+i\theta}\right)^{N'}$ whose square norm using the pulled back Hermitian metric is $\left| \widetilde F^*\bm\phi\right|_{\widetilde F^*h_\mathfrak{b}}^2= \left|\left( e^{t+i\theta}\right)^{N'} \right|^2 e^{u^\mathfrak{b}(t)-2N't}=e^{u^\mathfrak{b}(t)}$ converges to the constant function $1$ on any compact subset $\left[-T, T\right]\times S^1\subset X$.

It should be noticed that the Chern connection of the Hermitian metric $h_\infty$ on $\widetilde F^*L$ over $X$ is flat, but with connection form given by $A_{h_\infty} = - N'\left( dt + id\theta\right)=-N'\frac{dw}{w}$ in the cylindrical coordinate $(t,\theta)\in \mathbf{R}\times S^1$ and complex coordinate $w\in \mathbf{C}^*$ respectively. If we forget $g_\infty$, i.e. simply treat the complex geometry $(L_\infty, A_{h_\infty})$, the underline complex manifold $\mathbf{C}^*$ can be naturally compactified as $\mathbf{C}^*\cup\{0,\infty\}=\mathbf{P}^1$ but $A_{h_\infty}$ is extended as a singular connection $\overline{A_{h_\infty}}$ of Dirac monopole type on the nontrivially extended line bundle $\overline{L_\infty}$.

Next we take the base points to be $q_n=p_1=0\in \mathbf{P}^1$, the pointed Cheeger-Gromov limit is expected to be an asymptotically cylindrical metric. Consider the following ODE potentially arising from a ``Cheeger-Gromov'' limit of the above compact case when the volume tends to infinity:
\begin{equation}
\label{eqn:bubbleODE}
\left\{
\begin{array}{l}
u_{tt}
=
\frac{1}{\lambda}e^{2\alpha(u-e^u)}(e^u-1), 
 -\infty<t<+\infty\\
u(0) 
= -\mathfrak{c}, \;\;
\lim_{t\to -\infty} u_t(t)=2N', \;\; \lim_{t\to +\infty} u_t(t)=0.
\end{array}
\right.
\end{equation}
The difference in the initial value condition and asymptotic behavior is resulted from the fact that \eqref{eqn:ODE} describes Einstein-Bogomol'nyi equation on $S^2$ while the current ODE is aimed to describe such equation on $\mathbf{C}$ which is noncompact. The conditions in \eqref{eqn:bubbleODE} is a bit nonstandard since they are mixing initial value and boundary value. It is not straightforward to see what kind of initial conditions on $u(0), u_t(0)$ to guarantee the existence of solution satisfying the boundary condition $\lim_{t\to -\infty} u_t(t)=2N'$. On the other hand, if we replace the ``initial condition'' $u(0)=-\mathfrak{c}$ by another asymptotic boundary condition $\lim_{t\to +\infty} u_t(t)=0$, then the equation loses the uniqueness since it is autonomous.

Next, we explore the \emph{scaling symmetry} of the system on $\mathbf{C}$, or \emph{translation symmetry} on $\mathbf{R}\times S^1$ pertained to \eqref{eqn:bubbleODE}.  Let $\iota_{\epsilon}: \mathbf{C}\longrightarrow \mathbf{C}$ be the dilation map $w\mapsto \epsilon w=z$, which in the $(s,\theta)$ coordinate is given as $\gamma_\epsilon:\mathbf{R}\times S^1\longrightarrow \mathbf{R}\times S^1$ with $(s,\theta)\mapsto (s+\log \epsilon,\theta)=(t,\theta)$. Take any solution $u(t)$ to \eqref{eqn:bubbleODE}, the function $\gamma_\epsilon^*u(s)=u(s+\log \epsilon)$ is also solution to the ODE, but with initial condition $\gamma_\epsilon^*u(0)=u\left(\log\epsilon\right)$ and boundary condition $\lim_{s\to -\infty}\left( \gamma_\epsilon^*u\right)_s(s)=2N'$. 

Fix any $\mathfrak{c}>0$, for any $\mathfrak{b}\in \left(0,\mathfrak{c}\right)$, there is a unique point $t_{\mathfrak{b},\mathfrak{c}}\in \left(-\infty,0\right)$ such that $u^\mathfrak{b}(t_{\mathfrak{b},\mathfrak{c}})=-\mathfrak{c}$. By the autonomous property of the equation satisfied by $u^\mathfrak{b}$, we can ``translate'' $u^\mathfrak{b}$ to another solution of the ODE in \eqref{eqn:bubbleODE} with initial condition $-\mathfrak{c}$, i.e. we consider the function $\gamma_{e^{t_{\mathfrak{b},\mathfrak{c}}}}^*u^\mathfrak{b}(\cdot)=u^\mathfrak{b}\left(\cdot + t_{\mathfrak{b},\mathfrak{c}}\right)$ which is a solution to 

\begin{equation}
\label{eqn:ODEtranslate}
\left\{
\begin{array}{l}
u_{tt}
=
\frac{1}{\lambda_\mathfrak{b}}e^{2\alpha(u-e^u)}(e^u-1), \;\;
 -\infty<t<+\infty,\\
u(0) 
= -\mathfrak{c}, \;\; u_t'(0)>0.
\end{array}
\right.
\end{equation}
By the uniform convergence of $u^\mathfrak{b}$ on $\left[-T,T\right]$ to zero function for any fixed $T$ as $\mathfrak{b}\to 0+$, we know that $\lim_{\mathfrak{b}\to 0^+} t_{\mathfrak{b},\mathfrak{c}}=-\infty$. Fix any $T>0$, for $\mathfrak{b}$ small enough, $-t_{\mathfrak{b},\mathfrak{c}}>T$ and there exists $C_T>0$ such that 
\begin{itemize}
	\item  $\gamma_{e^{t_{\mathfrak{b},\mathfrak{c}}}}^*u^\mathfrak{b}|_{[-T,T]}$ is increasing;
	\item 
	\begin{equation}
		\label{est:ufunction}
	 -2N'T\leq \gamma_{e^{t_{\mathfrak{b},\mathfrak{c}}}}^*u^\mathfrak{b}\leq 0,\;\;  
		\left| \frac{d^i}{dt^i} \gamma_{e^{t_{\mathfrak{b},\mathfrak{c}}}}^*u^\mathfrak{b}\right|\leq C_T, \; i=1,2,3. 
		\end{equation}
	
\end{itemize}
By Arzela-Ascoli's Theorem, as $\mathfrak{b}\to 0$, the family has a subsequence that converges to a solution $u_{(\infty)}$ to \eqref{eqn:ODEtranslate} on $[-T,T]$ with parameter $\lambda_0$. By a standard diagonal argument (applied to a sequence $T_j\to +\infty$ ), we can assume $u_{(\infty)}$ is a solution defined on $\left(-\infty,+\infty\right)$. Moreover, $u_{(\infty)}$ is increasing and concave on $\mathbf{R}$ with $u_{(\infty)}(0)=-\mathfrak{c}$. If $u_{(\infty)}'(0)=0$, then $u_{(\infty)}'(t)\equiv 0$ on $[0,+\infty)$ and $u_{(\infty)}\equiv -\mathfrak{c}$ which is clearly a contradiction since the constant function is not a solution on $[0,+\infty)$. Similarly, if $u_{(\infty)}'(0)=2N'$ then $u_{(\infty)}'(t)\equiv 2N'$ on $(-\infty, 0]$ and $u_{(\infty)}''(t)\equiv 0$ which is clearly also a contradiction to the equation \eqref{eqn:ODEtranslate}. Therefore, we have 
\[
0<u_{(\infty)}'(t)<2N', \;\; t\in \mathbf{R}. 
\]
Similarly to the above argument, we actually can show that 
\[
u_{(\infty)}(t)<0, \;\; \lim_{t\to +\infty}u_{(\infty)}(t)=0,\;\; \lim_{t\to +\infty} u_{(\infty)}'(t)=0. 
\]
One property that is still missing in order $u_{(\infty)}$ defines a solution to the equation \eqref{eqn:bubbleODE} is that $$\lim_{t\to -\infty} u_{(\infty)}'(t)=2N'.$$

Another observation we made is that 
\[
2\pi N
=\frac{1}{2}\int_{\mathbf{P}^1} \left(1-|\bm\phi|_{h_\mathfrak{b}}^2\right)\omega_\mathfrak{b} 
\geq 
\frac{1}{2}\Vol_{g_\mathfrak{b}}  \left( { \left|\bm\phi\right|_{h_\mathfrak{b}}^2 \leq e^{-\mathfrak{c}}} \right)\cdot \left( 1- e^{-\mathfrak{c}}\right)
= \left(1-e^{-\mathfrak{c}}\right) \Vol_{g_\mathfrak{b}}\left( (-\infty, t_{\mathfrak{b},\mathfrak{c}}]\times S^1 \right).
\]
The volume of the region (which is a geodesic ball centered at $p_1=0$) is uniformly bounded and therefore its radius is also uniformly bounded by the linear volume growth estimate for metric with nonnegative curvature provided the family of metric balls $B_{g_{\mathfrak{b}}}\left(p_1, \delta\right)$ has uniform lower bound on their volume. To this end, we need to look at the family of metrics $g_\mathfrak{b}$ in a more intrinsic coordinate. Considering the $S^1$ invariance, the \emph{normal coordinate} $\mathfrak{r}$ based at the fixed point $p_1$ is a good choice. Recall that the notation $\Phi_\mathfrak{b}=\left| \bm\phi\right|_{h_\mathfrak{b}}^2$ is the \emph{state function}. The conformally rescaled metric $k_\mathfrak{b}=e^{2\alpha \Phi_\mathfrak{b}}g_\mathfrak{b}= \frac{1}{\lambda_\mathfrak{b}}e^{2\alpha u^\mathfrak{b}}\left(dt^2+d\theta^2\right)$ is of the form $d\mathfrak{r}^2 +\eta(\mathfrak{r})^2d\theta^2$ with $\eta(0)=0,\eta'(0)=1$ and $\eta(\mathfrak{r})>0$ for $\mathfrak{r}\neq 0$. The scalar curvature and covariant derivative of the scalar curvature
\[
S= -2\frac{\eta''}{\eta}, \;\; |\nabla S| = 2 \left|\left( \frac{\eta''}{\eta}\right)'\right|
\]
are shown bounded in \cite[Lemma 4.9]{FPY} by the constant $K=\max\{\alpha, \sqrt{\frac{3\alpha}{2}} \left(2\alpha+1\right)\}$. Then we have 
\begin{equation}
\label{eqn:bounds}
\begin{split}
-\frac{K}{2}\eta
& \leq \eta''\leq 0,\\
\left| \eta^{'''}\right|
& \leq \frac{K}{2}\left( \eta + \left|\eta'\right|\right). 
\end{split}
\end{equation}
In the current situation, $\eta_\mathfrak{b}= \frac{1}{\sqrt{\lambda_{\mathfrak{b}}}} e^{\alpha u^\mathfrak{b}}\leq \frac{1}{\sqrt{\lambda_0}}$ on its domain, and $0<\eta'_\mathfrak{b}(\mathfrak{r})\leq 1$ for $\mathfrak{r}\in \left(0, d_\mathfrak{b}\right)$ where $d_\mathfrak{b}=\int_{-\infty}^0 \frac{1}{\sqrt{\lambda_b}} e^{\alpha u^\mathfrak{b}(t)}dt$ is the distance from $p_1$ to the central equator of $\mathbf{P}^1$ under the metric $k_\mathfrak{b}$. The bounds in \eqref{eqn:bounds} implies that for the family of functions $\eta_\mathfrak{b_i}=\frac{1}{\sqrt{\lambda_{\mathfrak{b}_i}}} e^{\alpha u_{\mathfrak{b}_i}}$ where $\{ \mathfrak{b}_i\}_{i=1,2,\cdots}\subset \left(0,\mathfrak{c}\right)$ is a sequence converges to $0$, we can take a subsequence, still denoted by $\mathfrak{b}_i$, such that $\eta_{\mathfrak{b}_i}\longrightarrow \eta_{\mathfrak{0}}$ in $C^{2,\beta}$ sense (for any $\beta\in (0,1)$) on the interval $\mathfrak{r}\in \left[0,\mathfrak{r}_*\right]$ for any $\mathfrak{r}_*>0$ (using the fact $d_{\mathfrak{b}_i}\to +\infty$ established at the beginning of this subsection). By taking a sequence of $\mathfrak{r}_*$ diverging to $+\infty$ and using diagonal argument we can assume this sequence $\eta_{\mathfrak{b}_i}$ converges on $[0,+\infty)$ to $\eta_{\mathfrak{0}}$. A consequence of such convergence is that $\eta_\mathfrak{0}(0)=0, \eta_{\mathfrak{0}}'(0)=1$ and $\eta'_{\mathfrak{0}}(0)\geq 0$ on $[0,+\infty)$. There exists $\delta_0>0$ such that $\eta_{\mathfrak{0}}(\mathfrak{r})\geq \frac{1}{2}\mathfrak{r}$ for $\mathfrak{r}\in (0,\delta_0)$ and therefore $\frac{1}{\sqrt{\lambda_0}}\geq \eta_{\mathfrak{0}}(\mathfrak{r})\geq \frac{\delta_0}{2}$ for any $\mathfrak{r}\in \left[\delta,+\infty\right)$. 

For this subsequence, $k_{\mathfrak{b}_i}=e^{2\alpha\Phi_{\mathfrak{b}_i}}g_{\mathfrak{b}_i}=d\mathfrak{r}^2+\eta_{\mathfrak{b}_i}(\mathfrak{r})^2 d\theta^2$ converges in $C^{2,\beta}$ sense to a complete $S^1$ invariant Riemannian metric $k_{\mathfrak{0}} = d\mathfrak{r}^2 + \eta_{\mathfrak{0}}(\mathfrak{r})^2d\theta^2$ on $\mathbf{R}^2$ whose curvature satisfies $0\leq S_{k_{\mathfrak{0}}}\leq K$. The limit of $\eta_{\mathfrak{0}}(\mathfrak{r})$ as $\mathfrak{r}\to +\infty$ exists and is a positive number, thus $k_\mathfrak{0}$ is asymptotic to a cylindrical metric $d\mathfrak{r}^2+\eta_{\mathfrak{0}}(+\infty)^2d\theta^2$. It is not clear at this moment if $\eta_{\mathfrak{0}}(+\infty)=\frac{1}{\sqrt{\lambda_0}}$ holds, which is a piece of crucial information to the understanding of the convergence of the Hermitian metric $h_{\mathfrak{b}_i}$. 

Let $\mathfrak{r}_{\mathfrak{b}_i,\mathfrak{c}}= \int_{-\infty}^{t_{\mathfrak{b}_i,\mathfrak{c}}}\frac{1}{\sqrt{\lambda_{\mathfrak{b}_i}}} e^{\alpha u_{\mathfrak{b}_i}(t)}dt$ be the distance from $p_1$ to a point where $t=t_{\mathfrak{b}_i,\mathfrak{c}}$ under the metric $k_{\mathfrak{b}_i}$. According to the above volume estimate, for $i$ large enough
\begin{equation}
\begin{split}
\frac{2\pi N}{1-e^{-\mathfrak{c}}}
&\geq 
\Vol_{e^{2\alpha \Phi_{\mathfrak{b}_i}}g_{\mathfrak{b}_i}}
\left( 
(-\infty, t_{\mathfrak{b}_i,\mathfrak{c}}]\times S^1
\right)
=
2\pi \int_0^{\mathfrak{r}_{\mathfrak{b}_i,\mathfrak{c}}}  \eta_{\mathfrak{b}_i}(\mathfrak{r})d\mathfrak{r}\\
& \geq 2\pi \left( \mathfrak{r}_{\mathfrak{b}_i,\mathfrak{c}}-\delta_0\right)\frac{\delta_0}{4},
\end{split}
\end{equation}
where in the last inequality we use the fact that $\lim_{i\to +\infty} \eta_{\mathfrak{b}_i}(\delta_0)=\eta_{\mathfrak{0}}(\delta_0)\geq \frac{\delta_0}{2}$. This implies that $d_{g_{\mathfrak{b}_i}}\left(p_1, q\right)$ is uniformly bounded from above where $q\in \partial\left( |\bm\phi|_{h_{\mathfrak{b}_i}}^2\leq e^{-\mathfrak{c}}\right)\cap \{t<0\}$. On the other hand, since $\Phi_{\mathfrak{b}_i}(p_1)=0$ and $\Phi_{\mathfrak{b}_i}(q)=e^{-\mathfrak{c}}$, by the mean value inequality and the gradient estimate of the state function $\Phi_{\mathfrak{b}_i}$ (\cite[Corollary 4.6]{FPY}),
\[
e^{-\mathfrak{c}}\leq \sup_{\mathbf{P}^1}
\left|\nabla\Phi_{\mathfrak{b}_i}\right|_{k_{\mathfrak{b}_i}}\cdot d_{k_{\mathfrak{b}_i}}\left(p_1, q\right)
\leq \sqrt{\frac{3}{2\alpha}}\cdot d_{k_{\mathfrak{b}_i}}\left(p_1, q\right),
\]
therefore $d_{k_{\mathfrak{b}_i}}\left(p_1, q\right)=\mathfrak{r}_{\mathfrak{b}_i,\mathfrak{c}}$ is uniformly bounded from below. Geometrically, this says that the circle $t=t_{\mathfrak{b}_i,\mathfrak{c}}$ in $\mathbf{P}^1$ is at a controllable distance from $p_1$ under the metric $k_{\mathfrak{b}_i}$. 

Next, we want to show that the limit $u_{(\infty)}$ on $t\in (-\infty,+\infty)$ (obtained in analysis way) and the limit $\eta_{\mathfrak{0}}$ on $\mathfrak{r}\in [0,+\infty)$ (obtained in geometric way) could be patched together to give a global solution of \eqref{eqn:bubbleODE}. 

Consider the function $\psi_\mathfrak{b}:\left(-\infty, +\infty\right)\longrightarrow (0,+\infty)$ defined by
\[
\mathfrak{r}=\psi_\mathfrak{b}(t)
= 
\int_{-\infty}^t \frac{1}{\sqrt{\lambda_\mathfrak{b}}} e^{\alpha u^\mathfrak{b}(\tau)}d\tau, 
\]
then $\eta_\mathfrak{b}(\mathfrak{r})= \frac{1}{\sqrt{\lambda_\mathfrak{b}}} e^{\alpha u^\mathfrak{b}\left( \psi_\mathfrak{b}^{-1}(\mathfrak{r})\right) }$ and  for $t\in (-\infty,+\infty)$,
\begin{equation}
\label{eqn:relations}
\gamma_{e^{t_{\mathfrak{b},\mathfrak{c}}}}^*u^\mathfrak{b}(t)
=\frac{1}{\alpha} \log \sqrt{\lambda_\mathfrak{b}} \eta_\mathfrak{b}\left( \psi_\mathfrak{b}\left( t+t_{\mathfrak{b},\mathfrak{c}}\right)\right)
=\frac{1}{\alpha} \log \sqrt{\lambda_\mathfrak{b}} \eta_\mathfrak{b}\circ \gamma_{e^{t_{\mathfrak{b},\mathfrak{c}}}}^*\psi_\mathfrak{b}(t).
\end{equation}
 The function  $\gamma_{e^{t_{\mathfrak{b},\mathfrak{c}}}}^*\psi_\mathfrak{b}(t)$ from $(-\infty, +\infty)$ to $(0,+\infty)$ can be viewed as a transition map between the two coordinate charts $(t,\theta)$ and $(\mathfrak{r},\theta)$ of the manifold $X=\mathbf{R}^2$. Since for any fixed $T>0$,
\begin{itemize}
	\item $\gamma_{e^{t_{\mathfrak{b},\mathfrak{c}}}}^*\psi_\mathfrak{b}(0)= \mathfrak{r}_{\mathfrak{b},\mathfrak{c}}$ is uniformly bounded; 
	\item 
	$\frac{d}{dt}\gamma_{e^{t_{\mathfrak{b},\mathfrak{c}}}}^*\psi_\mathfrak{b}(t)= \frac{1}{\sqrt{\lambda_\mathfrak{b}}}e^{\alpha\cdot \gamma_{e^{t_{\mathfrak{b},\mathfrak{c}}}}^*u^\mathfrak{b}(t)}$ is uniformly bounded above by $\frac{1}{\sqrt{\lambda_\mathfrak{b}}}$ andy below by $\frac{1}{\sqrt{\lambda_\mathfrak{b}}}e^{-T}$ according to  \eqref{est:ufunction} on $[-T,T]$;
	\item  
	$\frac{d^2}{dt^2}\gamma_{e^{t_{\mathfrak{b},\mathfrak{c}}}}^*\psi_\mathfrak{b}(t)=\alpha \cdot \gamma_{e^{t_{\mathfrak{b},\mathfrak{c}}}}^*\psi_\mathfrak{b}(t)\cdot \frac{d}{dt}\gamma_{e^{t_{\mathfrak{b},\mathfrak{c}}}}^*u^\mathfrak{b}(t)$ and similarly any higher order derivatives of $\gamma_{e^{t_{\mathfrak{b},\mathfrak{c}}}}^*\psi_\mathfrak{b}(t)$ are uniformly bounded on $[-T,T]$, 
\end{itemize}
together with a diagonal argument we can take a subsequential limit $ \widetilde \psi_\mathfrak{0}$ of $\gamma_{e^{t_{\mathfrak{b},\mathfrak{c}}}}^*\psi_\mathfrak{b}$ on $(-\infty,+\infty)$, whose convergence is in $C^3$ sense on any subinterval $[-T,T]\subset (-\infty,\infty)$. By the previous two convergence results, one for $\eta_{\mathfrak{b}}$ and one for $\gamma_{e^{t_{\mathfrak{b},\mathfrak{c}}}}^*u^\mathfrak{b}(t)$, and the relation \eqref{eqn:relations} between the two functions, we conclude that 
\begin{equation}
u_{(\infty)}(t)
= 
\frac{1}{\alpha}\log \sqrt{\lambda_0} \eta_\mathfrak{0}\circ \widetilde \psi_\mathfrak{0}(t)
\end{equation}
where $\widetilde \psi_\mathfrak{0}(t)=\int_{-\infty}^t \frac{1}{\sqrt{\lambda_0}} e^{\alpha u_{(\infty)}(\tau)}d\tau$. 

In conclusion, the smooth Riemannian metric 
\[
g_{(\infty)}
=
\frac{1}{\lambda_0}e^{2\alpha \left( u_{(\infty)}(t)-e^{u_{(\infty)}(t)}\right)} \left( dt^2+d\theta^2\right)
\]
on $\mathbf{R}\times S^1$ coincides with the $C^{2,\beta}$ Riemannian metric 
\[
e^{-2\alpha  e^{u_{(\infty)} \left( \widetilde\psi_\mathfrak{0}^{-1}(\mathfrak{r}) \right)} }\left( d\mathfrak{r}^2 + \eta_\mathfrak{0}(\mathfrak{r})^2 d\theta^2\right)
\]
on $\mathbf{R}^2$ in their intersecting region, i.e. $\mathbf{R}^2\backslash \{0\}$. As a consequence, this shows $\lim_{t\to -\infty} u_{(\infty)}'(t)=2N'$. As in the derivation of \eqref{est:gradient}, we get the formula
\[
4N'^2-\left(\frac{du_{(\infty)}}{dt}\right)^2
=
\frac{1}{\alpha\lambda_0} e^{2\alpha\left( u_{(\infty)}(t)-e^{u_{(\infty)}(t)}\right)}
\] 
and the rest arguments as in section \eqref{section:regularity} shows that this Riemannian metric is actually smooth on $\mathbf{R}^2$, and is actually K\"ahler with respect to the standard complex structure on $\mathbf{R}^2$. Hence, we obtain an Einstein-Bogomol'nyi metric $\left( g_{(\infty)}, h_{(\infty)}\right)$ on $\mathbf{C}$.

By multiplying $\frac{d}{dt}u^{(\infty)}$ to both sides of  \eqref{eqn:ODEtranslate} (with parameter $\lambda_0$) and integrating on $(-\infty,0]$, we obtain that 
\begin{equation}
\label{eqn:determinedderivative}
\frac{d}{dt}|_{t=0}u_{(\infty)}
= 
\sqrt{ 4N'^2 - \frac{1}{\alpha\lambda_0}e^{-2\alpha\left(\mathfrak{c}+e^{-\mathfrak{c}}\right)}}  
= 
2N'\sqrt{ 1 - e^{2\alpha\left(1-\mathfrak{c}-e^{-\mathfrak{c}}\right)}}. 
\end{equation}
The solution $u_{(\infty)}$ we constructed above depends on the particular choice of $\mathfrak{c}\in (0,+\infty)$. However, since the derivative of the solution at $t=0$ is determined by the value of the solution at $t=0$, different choices of $\mathfrak{c}$ gives solutions which coincide with each other under suitable translation. In other words, there is geometrically only \emph{one} Einstein-Bogomol'nyi metric (as a solution of \eqref{eqn:bubbleODE}) constructed in this way. 

Similar to \eqref{eqn:determinedderivative}, by integrating over $[t,+\infty)$ we obtain that $u=u_{(\infty)}$ is a solution to the first order ODE 
\begin{equation}
\frac{du}{dt}
= 
2N'\sqrt{ 1- e^{2\alpha\left( 1+ u-e^u\right)}}
\end{equation}
on $\mathbf{R}$ together with the condition $u(0)=-\mathfrak{c}, \lim_{t\to -\infty} u(t)=-\infty$. By Lagrange Mean Value Theorem,  $\forall t\in [0,+\infty)$, there exists $-2\alpha \left( e^{u(t)}-(1+u(t))\right)<\xi<0$ and $u(t)<\eta<0$ such that
\begin{equation*}
\begin{split}
\frac{du}{dt}
& =
2N'\sqrt{e^\xi \cdot 2\alpha \left(e^u-(1+u)\right)}= 
2N'\sqrt{ \alpha e^{\xi+\eta} u^2}.
\end{split}
\end{equation*}
It follows that on $[0,+\infty)$
\begin{equation*}
2N'\sqrt{\alpha} e^{-\alpha\left( e^{-\mathfrak{c}}-(1-\mathfrak{c})\right)-\frac{\mathfrak{c}}{2}}\left(-u\right)
\leq\frac{du}{dt}
\leq 2N'\sqrt{\alpha}\left(-u\right), 
\end{equation*}
and as a consequence for any $t\in [0,+\infty)$,
\begin{equation}
\mathfrak{c} e^{-2N'\sqrt{\alpha}t}
\leq 
-u(t)
\leq 
\mathfrak{c} e^{-2N'\sqrt{\alpha}e^{-\alpha\left( e^{-\mathfrak{c}} -(1-\mathfrak{c})\right)-\frac{\mathfrak{c}}{2}}t}.
\end{equation}

Denote $g_C=2N'\left(dt^2+d\theta^2\right)$ the flat metric on the cylinder $\mathbf{R}\times S^1$ with circumstance $4N'\pi$. According to the formula of $g_{(\infty)}$, 
\begin{equation*}
\begin{split}
g_{(\infty)} - g_C
& = 
 \left( e^{-2\alpha\left( e^{u} - (1+u) \right)}  - 1 \right)g_C\\
 iF_{h_{(\infty)}}
 & = 
 \frac{1}{2}\left( 1- e^u\right) \omega_{(\infty)}. 
 \end{split}
\end{equation*}
This shows that $g_{(\infty)}-g_C=O\left(u^2\right)$ and $iF_{h_{(\infty)}}=O(u)$ as $t\to +\infty$. Inductively, one shows that $\left|\nabla_{g_C}^\ell\left( g_{(\infty)} - g_C\right)\right|_{g_C}=O(u^2)$ and $\left| \nabla_{g_C}^\ell F_{h_{(\infty)}}\right|_{g_C}=O(u)$ for each $\ell\geq 1$, i.e. $g_{(\infty)}$ is asymptotic to $g_C$ at \emph{exponential} rate, and $iF_{h_{(\infty)}}$ and its higher derivatives are \emph{exponentially decaying}. 

Direct computation
\[
\int_{\mathbf{C}} iF_{h_{(\infty)}} 
= 
\lim_{t_1\to 0, t_2\to +\infty} \int_{[t_1,t_2]\times S^1} - i\partial\bar\partial u_{(\infty)}
= 
\lim_{t_1\to 0, t_2\to +\infty}\int_{t_1}^{t_2}\int_0^{2\pi} - \frac{1}{2}u_{(\infty)}''(t)dt\wedge d\theta
= 
2\pi N'
\]
tells us that the total string number of this Einstein-Bogomol'nyi metric is $N'$. The fact that $e^{u_{(\infty)}}$ is asymptotic to constant $1$ means the system is asymptotically \emph{superconducting} as $\bm\phi$ is the order parameter and $|\bm\phi|_{h}^2$ represents the density of Cooper pairs responsible for superconducting.

\begin{remark}
Yang \cite{Yang4} showed the nonexistence of $S^1$ symmetric Einstein-Bogomol'nyi metrics with all strings located at one point on $\mathbf{P}^1$. This is not a contradiction to what we obtained here since the Einstein-Bogomol'nyi metric constructed above cannot be compactified to a smooth Einstein-Bogomol'nyi metric on $\mathbf{P}^1$ whose Higgs field only vanishes at one point. This asymptotically cylindrical Einstein-Bogomol'nyi metric obtained here was originally found by Linet \cite{Linet}, and the way of presentation here is aimed to make it fit more naturally in the study of the moduli space of Einstein-Bogomol'nyi metrics on compact surface.
\end{remark}

\begin{remark}
	In analogy to the remark about the dependence of $\lambda$ and $V$ in the case $\bm\phi$ is stable, we have a conjectured relation of $V$ and $\lambda$ for strictly polystable $\bm\phi$ (see Figure \ref{fig:temper-strictly-polystable}). This graph is based on the study of $S^1$ symmetric solutions. 
	\begin{figure}[htb]
		\label{fig:temper-strictly-polystable}
		\begin{tikzpicture}[samples=200,domain=0:4]
		\draw[->] (-1.2,0) -- (5,0) node[right] {$volume\;\; V$};
		\draw[dashed] (-1, 2)  node[left] {} -- (5,2);
		\draw[fill=orange,color=black] (-1.1,2.2)  node[left]{critical $\lambda_c=\frac{1}{Ne^{2\alpha}}$};
		\draw[fill=orange,color=blue] (0,0) circle (1pt) node[below]{$\underline V$};
		\draw[fill=orange,color=black] (-1.2,0)  node[below]{$o$};
		\draw[->] (-1,-0.2) -- (-1,3) node[above] {$temper\;\;\lambda$};
		\draw[color=red, <-, domain=0:4.8]  plot (\x, 2-2/(\x+1) node[right] {$\lambda$};
		\end{tikzpicture}
		\caption{Conjectured relation of $V$ and $\lambda$ for strictly polystable $\bm\phi$}
	\end{figure}
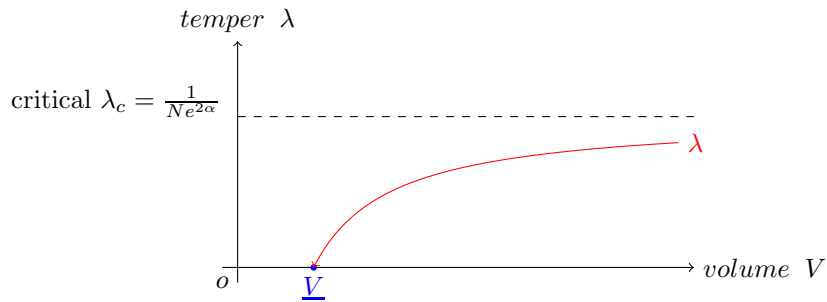
\end{remark}

\subsection{The solution of Chen-Hastings-McLeod-Yang on $\mathbf{C}$}
We are interested in constructing an Einstein-Bogomol'nyi solution with asymptotically conical behavior of the metric, which 
represents a smoothing of the \emph{delta-conical} EB solution.

Chen-Hastings-McLeod-Yang \cite[Theorem 3.2]{CHMcY}, studied the following  ODE (with parameter $\lambda>0$) for $u:\mathbf{R}^+\to \mathbf{R}$:
\begin{equation}
\label{eqn:CHMY}
\begin{split}
u_{rr}+ \frac{1}{r} u_r
& = 
-r^{-2aN}f(u,a,\lambda)\\
u(r)
& = 
2N \log r + s + o(1) , \;\; \text{as }r\to 0.
\end{split}
\end{equation}
where $f(u,a, \lambda)=\frac{1}{\lambda} e^{a(u-e^u)}(1-e^u)$. In the situation $aN\in [0,1)$, it is showed that there exists a unique $s_*\in \mathbb{R}$ such that the (unique existing global) solution to the above system satisfies 
\[
\lim_{r\to +\infty} u(r) = 0, 
\] 
Actually, this function is proved (\cite[Equation (5.11), (5.12)]{CHMcY}) to further satisfy 
\[
u(r)= O(r^{-\kappa}), \;\; u'(r) = O(r^{-\kappa}) , \; r\to +\infty
\]
for any $\kappa>0$.  Combined with equation \eqref{eqn:CHMY}, it holds that 
\begin{equation*}
\frac{d^\ell}{dr^\ell}u(r)
= O(r^{-\kappa}), \;\; r\to +\infty
\end{equation*}
for any given $\kappa>0$.

Under the standard trivialization of the trivial line bundle $\mathcal{O}$, define a holomorphic section $\bm\phi$ which is represented by $z^N$ in this trivialization. Define a Hermitian metric $h$ (which is represented by a positive function $H$ on $\mathbb{C}$):

\begin{equation}
\begin{split}
H (z)
&= 
\frac{|\bm\phi|_h^2}{|\bm\phi|^2}
=\frac{e^u}{|z^N|^2}
= r^{-2N} e^{u(r)} 
\sim |z|^{-2N}\;\;,\;\; z\to \infty.\\
H(0)
& = 
e^s.
\end{split}
\end{equation}
In the meantime, define a Riemannian metric:
\begin{equation}
\begin{split}
g
& = 
\frac{1}{\lambda} r^{-2aN}e^{a(u-e^u)}g_{euc}, 
\end{split}
\end{equation}
which is conformal to the standard Euclidean metric and is complete. This implies $g$ is K\"ahler with respect to the standard complex structure $J_0$ on $\mathbf{R}^2$, so denote its K\"ahler form by $\omega$. 

The pair $(\omega, h)$ satisfies 
\begin{equation}
\left\{
\begin{array}{ll}
iF_h + \frac{1}{2}\left( |\bm\phi|_h^2 -1\right)\omega 
& =0, \\
\text{Ric }\omega - a i\partial\bar\partial |\bm\phi|_h^2 - a iF_h 
& = 0, 
\end{array}
\right.
\end{equation}
which is exactly the Einstein-Bogomol'nyi equation we've considered with $\alpha=\frac{a}{2}$ and $\tau=1$.

The scalar curvature of $\omega$ is 
\begin{equation}
\begin{split}
S_\omega 
& =2\alpha \left|\nabla^{1,0}\bm\phi\right|_h^2 + \alpha\left(\tau-\left| \bm \phi\right|_h^2\right)^2
=
\alpha \frac{ \left|\nabla \left|\bm\phi\right|_h^2\right|_g^2}{\left|\bm\phi\right|_h^2}
+ 2\left( \tau-\left|\bm\phi\right|_h^2\right)^2 \\
& = 
\alpha \frac{\left|\nabla e^u\right|_g^2}{e^u}
+ 
\alpha \left(\tau - e^u\right)^2\\
& = 
\frac{a}{2}\left(  e^u \left|\nabla u\right|_g^2 
+ \left(1-e^u\right)^2
\right). 
\end{split}
\end{equation}
It follows from the decaying estimate of $u$ and $u'$ above that 
\[
S_\omega= O(r^{-\kappa}), \;\;
 \left| iF_h\right|_h= O(r^{-\kappa}), \;\; 
 \left| \bm \phi\right|_h^2 = 1 + O(r^{-\kappa}).
\]

Denote $\widehat g_{(\beta)} = \mathrm{d}\mathfrak{r}^2 +\beta^2 \mathfrak{r}^2\mathrm{d}\theta^2$ the standard flat cone metric with angle $2\pi\beta= 2\pi (1-aN)$ on $\mathbf{C}$. Define a diffeomorphism (inclusion)
\[
\begin{split}
F: \mathbf{C}^*
& \longrightarrow \mathbf{C}\\
\left(\mathfrak{r},\theta\right) 
& \mapsto 
\left( 
r,\theta
\right)
= 
\left( 
\left[
\frac{\lambda^\frac{1}{2}(1-aN)}{e^{-\frac{a}{2}}}
\right]^\frac{1}{1-aN} \mathfrak{r}^\frac{1}{1-aN} 
\right), 
\end{split}
\] 
then $F^*g= e^{a\left( \tilde u(\mathfrak{r}) - e^{\tilde u\left(\mathfrak{r}\right)}+1\right)}
\left[ d\mathfrak{r}^2 + \beta^2 \mathfrak{r}^2 d\theta^2\right]$ with $\tilde u(\mathfrak{r}) = u(r)$. Using the above decaying estimate on $u^{(\ell)}(r)$ it is not difficult to show that for any $\mathfrak{r}_0>0$, 
\begin{equation}
\begin{split}
\frac{\partial^\ell}{\partial\mathfrak{r}^\ell} 
\left( F^*g- \widehat g_{(\beta)}\right)
= 
O(\mathfrak{r}^{-\kappa}), \;\; \mathfrak{r}>\mathfrak{r}_0, 
\end{split}
\end{equation}
i.e. $g$ is asymptotic to $\widehat g_{(\beta)}$ faster than any polynomial rate. Since $g$ is conformal to $g_{euc}$, let $J_\mathbf{C}$ be the standard complex structure on $\mathbf{C}$ sending $dr$ to $-rd\theta$, then $\left(g,J_\mathbf{C}\right)$ is K\"ahler whose K\"ahler form is denoted as $\omega$, then $F^*J_\mathbf{C}= J_{\widehat{\mathbf{C}}_{(\beta)}}$ where 
$J_{\widehat{\mathbf{C}}_{(\beta)}}$ sending $d\mathfrak{r}$ to $-\beta\mathfrak{r}d\theta$ is the complex structure on $\mathbf{C}^*$ making $\widehat g_{(\beta)}$ K\"ahler. Similarly to $F^*g$ we have $F^*\omega = e^{a\left( \tilde u(\mathfrak{r}) - e^{\tilde u\left(\mathfrak{r}\right)}+1\right)} \widehat\omega_{(\beta)} $, and 
\[
\frac{\partial^\ell}{\partial\mathfrak{r}^\ell} 
\left( F^*\omega- \widehat \omega_{(\beta)}\right)
= 
O(\mathfrak{r}^{-\kappa}), \;\; \mathfrak{r}>\mathfrak{r}_0.
\]

We want to turn the the behavior of the structures on the holomorphic vector bundles. Define a bundle map $\widetilde F: \mathbf{C}^*\times \mathbf{C}\longrightarrow \mathbf{C}\times \mathbf{C}$ sending $\left( (\mathfrak{r},\theta) , \xi\right)$ to $\left( F(\mathfrak{r},\theta), \left[ \left( \lambda^\frac{1}{2}\beta e^\frac{a}{2} \right)^\frac{1}{\beta} \mathfrak{r}^{\frac{1}{\beta} -1}\right]^N\cdot \xi\right)=\left( F(\mathfrak{r},\theta), \xi'\right)=\left( F(w), \left(\frac{F(w)}{w}\right)^N\cdot \xi\right)$, where the second components are the fiber coordinates of the trivial holomorphic line bundles on the corresponding base manifolds $\mathbf{C}^*$ and $\mathbf{C}$. The holomorphic section $\bm\phi$ is pulled back to $\widetilde F^*\bm\phi$, which at $w=\left(\mathfrak{r},\theta\right)$ takes the value $\widetilde F^{-1}\left( \phi \left( F(\mathfrak{r},\theta) \right)\right)= \left( \frac{w}{F(w)}\right)^N \cdot \phi \left( F(w)\right)=w^N$ under the above trivialization. Since $\widetilde F$ is a holomorphic map, it pulls the Hermitian metric $h$ back to a Hermitian metric $\widetilde F^*h$ and Chern connection to the Chern connection. The norms of the holomorphic sections satisfy 
\begin{equation}
\left| \widetilde F^*\bm\phi\right|_{\widetilde F^*h}^2 
= 
F^*\left|\bm\phi\right|_h^2
=
e^{\tilde u(\mathfrak{r})}
= 
1+ O(\tilde u(\mathfrak{r})),
\end{equation}
and the curvature forms satisfy 
\begin{equation}
\begin{split}
iF_{\widetilde F^*h}
& =
 F^*\left( iF_h\right) \\
&= 
F^*\left( \frac{1}{2}\left(1-|\bm\phi|_h^2\right)\omega\right)\\
& = 
\frac{1}{2}\left( 1- e^{\tilde u(\mathfrak{r})}\right) e^{a\left( \tilde u(\mathfrak{r})-e^{\tilde u(\mathfrak{r})}+1\right)} \widehat \omega_{(\beta)} \\
& = 
O\left( \tilde u(\mathfrak{r})\right)\cdot\widehat \omega_{(\beta)}.
\end{split}
\end{equation}
We define a Hermitian metric $\widehat{h}_{(N)}$ on the trivial bundle over $\mathbf{C}^*$ by $\widehat h_{(N)}(w):=\frac{1}{\left|w^N\right|^2}$, then obviously its curvature form is identically $0$ on $\mathbf{C}^*$. However, it can be viewed as a ``singular'' Hermitian metric over $\mathbf{C}=\mathbf{C}^*\cup \{0\}$ whose curvature form satisfies $iF_{\widehat h_{(N)}}= - i\partial\bar\partial \log \frac{1}{|w^N|^2}=2\pi N[0]
$ on $\mathbf{C}$ as a current. Under those definitions, the Einstein-Bogomol'nyi solution $(\omega,h)$ is asymptotic to the \emph{singular} Einstein-Bogomol'nyi metric $\left(\widehat \omega_{(\beta)}, \widehat h_{(N)}\right)$ with the Higgs field $\widehat{\bm\phi}=w^N$.

Parallel to the above above asymptotic analysis, we can do the analysis on the \emph{blowing down} limit of the geometric structure $(g,h)$. Let $\epsilon_i$ be a positive sequence converging to $0$, and $\iota_{\epsilon_i^{-1}}:\widehat{\mathbf{C}}_{(\beta)}\longrightarrow \widehat{\mathbf{C}}_{(\beta)}$ be the dilation map $(\mathfrak{s},\theta)\mapsto \left( \epsilon_i^{-1}\mathfrak{s},\theta\right)$. Under this map, $\iota_{\epsilon_i^{-1}}^*\widehat\omega_{(\beta)}=\epsilon_i^{-2}\widehat\omega_{(\beta)}$. Respectively, the bundle map $\widetilde\iota_{\epsilon_i^{-1}}: \widehat{\mathbf{C}}_{(\beta)}\times \mathbf{C}\longrightarrow \widehat{\mathbf{C}}_{(\beta)}\times \mathbf{C}$ sending $\left( (\mathfrak{s},\theta), \xi''\right)$ to $\left( (\epsilon_i^{-1}\mathfrak{s},\theta), \epsilon_i^{-N}\xi'' \right)$ is holomorphic.  Moreover, $\widetilde\iota_{\epsilon_i^{-1}}^*\widehat{\bm\phi}(u)=u^N=\widehat{\bm\phi}(u)$ and $\widetilde\iota_{\epsilon_i^{-1}}^*\widehat h_{(N)}(u)=\frac{1}{|u^N|^2}= \widehat h_{(N)}(u)$ where $u=\mathfrak{s}e^{i\theta}$, i.e. the Higgs field $\widehat{\bm\phi}$ and the singular Hermitian metric $\widehat{h}_{(N)}$ are both dilation invariant. Under the trivialization we are using, the connection form of $\widehat h_{(N)}$ is given by $A_{\widehat h_{(N)}}= \partial \log \widehat h_{(N)}=-N\frac{dw}{w}$.

Let $F_i=F\circ \iota_{\epsilon_i^{-1}}$ and $\widetilde F_i = \widetilde F\circ \widetilde \iota_{\epsilon_i^{-1}}$, then $F_i^*\left( \epsilon_i^2 g\right)
=
e^{a(\tilde u_i(\mathfrak{s}) - e^{\tilde u_i(\mathfrak{s})}+1)}
\left[
\mathrm{d}\mathfrak{s}^2 + \beta^2 \mathfrak{s}^2\mathrm{d}\theta^2
\right]$ 
where 
\[
\tilde u_i(\mathfrak{s})
=
u\left(
\left[
\lambda^\frac{1}{2}\beta e^{\frac{a}{2}}
\right]^\frac{1}{\beta} \left( \epsilon_i^{-1}\mathfrak{s}\right)^\frac{1}{\beta} 
\right). 
\]

Using the big ``$O$'' notation to mean a bounded quantity independent of $i$. From the formula, we see that for any rate $\kappa>0$ and radius $\mathfrak{s}_0>0$,
\begin{equation}
\begin{split}
F_i^*\left( \epsilon_i^2g\right) - \widehat g_{(\beta)}
& = 
O(\tilde u_i(\mathfrak{s})^2) 
= 
O\left( \epsilon_i^\frac{\kappa}{\beta} \mathfrak{s}^{-\frac{\kappa}{\beta}}\right), \;\; \mathfrak{s}\geq \mathfrak{s}_0,\\
\frac{\partial^\ell}{\partial\mathfrak{s}^\ell} 
\left( F_i^*\left(\epsilon_i^2g\right)- \widehat g_{(\beta)}\right)
& = 
O\left( \epsilon_i^\frac{\kappa}{\beta} \mathfrak{s}^{-\frac{\kappa}{\beta}}\right), \;\; \mathfrak{s}\geq\mathfrak{s}_0.
\end{split}
\end{equation}
Moreover, for any $\kappa>0$ and radius $\mathfrak{s}_0>0$,
\begin{equation}
\begin{split}
F_i^*\bm\phi 
& =  \widehat{\bm\phi}, \;\; 
F_i^*\left|\bm\phi\right|_h^2 
= 
\left| \widetilde F_i^*\bm\phi\right|_{\widehat F_i^*h}^2 = e^{\tilde u_i(\mathfrak{s})}
= 1 + O(\tilde u_i(\mathfrak{s}))
= 1+ O\left( \left( \epsilon_i^{-1}\mathfrak{s}\right)^{-\frac{\kappa}{\beta}}\right),\\
F_i^*\left( iF_h\right)
& = 
i F_{\widetilde F_i^*h}
= 
\frac{1}{2\epsilon_i^2}\left( 1- e^{\tilde u_i(\mathfrak{s})}\right) \cdot 
e^{a\left( \tilde u_i(\mathfrak{s})-e^{\tilde u_i(\mathfrak{s})}+1\right)}\widehat \omega_{(\beta)}
= 
O\left( \epsilon_i^{-2} \tilde u_i(\mathfrak{s})\right)
= 
O\left( \epsilon_i^{-2} \left(\epsilon_i^{-1}\mathfrak{s}\right)^{-\frac{\kappa}{\beta}}\right),\\
\frac{\partial^\ell}{\partial \mathfrak{s}^\ell}  F_i^*\left( iF_h\right) 
& = 
O\left( \epsilon_i^{-2} \left(\epsilon_i^{-1}\mathfrak{s}\right)^{-\frac{\kappa}{\beta}}\right).
\end{split}
\end{equation}

This in particular proves that the tangent cone at $\infty$ of the metric $g$ exists and is unique, which is $\widehat g_{(\beta)}$. In the meantime, the flux current of $h$ measured under the rescaled underlying metric is 
\begin{equation}
iF_{h}
=
\frac{1}{2\epsilon_i^2}(1-e^{u_i(\mathfrak{r})}) \text{dvol}_{g_i}
\longrightarrow 2\pi N[\bm 0]. 
\end{equation}

By the fact that $u$ defines a topological solution, we have the total magnetic flux
\begin{align*}
\int_{\mathbf{C}} iF_h
& = 
\int_{\mathbf{C}\backslash\{\bm 0\}} 
- i\partial\bar\partial \log |\bm\phi|_h^2 \\
& = 
\frac{1}{2}
\int_{\mathbf{R}^2\backslash\{(0,0)\}} \Delta_{g_{euc}} u\; \text{dvol}_{euc}\\
& = 
\pi \int_0^{+\infty} r^{1-2aN} f(u,a,\lambda)\mathrm{d}r\\
& = 
2\pi N.
\end{align*}
This can also be verified directly by Stokes Theorem. For any $R>\delta>0$, 
\begin{equation}
\begin{split}
\int_{B(0,R)\backslash B(0,\delta)} iF_h 
& = 
\int_{B(0,R)\backslash B(0,\delta)} -i\partial\bar\partial \log |\bm\phi|_h^2\\
& = 
\int_{\partial B(0,R)}-i\bar\partial u + \int_{B(0,\delta)} i\bar\partial \left( \log \left| z^N\right|^2 +\log h(z)\right).
\end{split}
\end{equation}
The first term is controlled by $|\nabla_{g_{euc}} u|\cdot 2\pi R = O(R^{-\kappa+1})$, the second term is equal to $2\pi N$ and the third term is controlled by $|\nabla_{g_{euc}} \log h| \cdot 2\pi \delta =O(\delta)$. By choosing $\kappa>1$ and letting $R\to +\infty, \delta\to 0$, this also concludes that the total magnetic flux is $2\pi N$. Similarly, we can compute the total scalar curvature: 
\begin{equation}
\begin{split}
\int_{B(0,R)} S_\omega \omega 
& = \int_{B(0,R)} \text{Ric }\omega \\
& = \int_{B(0,R)} ai\partial\bar\partial |\bm\phi|_h^2 + a iF_h\\
& = a\int_{\partial B(0,R)}  i e^u \bar\partial u 
+ a\int_{B(0,R)} iF_h.
\end{split}
\end{equation}
The first term is controlled by $e^u|\nabla_{g_{euc}} u|\cdot 2\pi R= O(R^{-\kappa+1})$, and by letting $R\to +\infty$ there holds that 
\[
\int_{\mathbf{C}} S_\omega \omega = 2\pi aN=2\pi (1-\beta).
\]

Geometrically, this shows that for the Einstein-Bogomol'nyi metric $\left(\omega, h, \bm\phi\right)$ with $S^1$ symmetry constructed by \cite{Linet,Yang}, the underlying Riemannian metric $\omega$ is asymptotically conical with conical angle $2\pi \left(1- aN\right)$, and $h$ is a Gaussian type Hermitian metric.  

We conjecture that the large volume limits of a family of Einstein-Bogomol'nyi metrics with fixed stable Higgs field, the Cheeger-Gromov limit based at $p_j$ will converges to one of the solution of Chen-Hasting-McLeod-Yang.

\section{Further discussion about Moduli space}

One of the main goals in the study of Einstein-Bogomol'nyi metrics is to understand the structure of the moduli space $\mathfrak{M}_{EB}\left(L,\tau\right)$ and $\mathfrak{M}_{EB}\left(L,\tau;V\right)$. It is now very reasonable to believe/conjecture the uniqueness (modulo automorphisms) of Einstein-Bogomol'nyi metrics for fixed $\bm\phi$ and $V$ as it is evidenced by the discussion of \cite[Section 5.3]{Al-Ga-Ga-P} and Theorem \ref{thm:isolated}.

The large volume limit in section \ref{sect:large} shows a very interesting link between the moduli space of Einstein-Bogomol'nyi metrics and the moduli space of \emph{Euclidean cone metric} on $S^2$ with total volume $2\pi$ and suitable cone angles, which is probably easier to understand. Moduli space of Euclidean cone metrics on $S^2$ with designated curvatures (or equivalently with designated cone angles) at more than $3$ points was studied by Thurston \cite{Th}. For instance, the moduli space has a natural K\"ahler metric, making it into a locally complex hyperbolic manifold, and the metric-completion of this natural K\"ahler metric gives rise to a complex hyperbolic cone-manifold whose singularities corresponds to collisions of the cone points \cite[Theorem 0.2]{Th}. See the related discussion in \cite{KN}. We expect that understanding towards such more classical moduli space would shed lights on the study of $\mathfrak{M}_{EB}(L,\tau)$. 

Any un-ordered tuple $\left( n_1, n_2, \cdots, n_d\right)\in \mathbf{N}_+^d$ satisfying $n_1+n_2+\cdots + n_d=N$ and $2n_j<N$ for each $j\in \{1, 2, \cdots, d\}$ is called \emph{a stable partition} of $N$ with length $d$. The particular partition $(N', N')\in \mathbf{N}_+^2$ with $N=2N'$ is called the \emph{strictly polystable partition} of $N$. A partition which is either stable or strictly polystable is called a polystable partition. Any holomorphic section $\bm\phi$ is said to be compatible with a partition $\mathfrak{p}$ if the tuple of multiplicities of the zeros of $\bm\phi$ is equal to $\mathfrak{p}$. Under this definition, a stable Higgs field is precisely a holomorphic section compatible with a stable partition, and a strictly polystable Higgs field is precisely a holomorphic section compatible with the strictly polystable partition.

For each polystable partition $\mathfrak{p}=\left( n_1, \cdots, n_d\right)$, the corresponding tuple of apex curvatures $$\frac{4\pi}{N}\left(n_1, \cdots, n_d\right)\in \left(0, 2\pi\right)^d$$ 
satisfy the numerical condition of Theorem 0.2 of \cite{Th} (actually Thurston was studying the more general situation of tuples of real numbers $\left(k_1, \cdots, k_d\right)\in \left(0,2\pi\right)^d$ with $k_1+\cdots+k_d=4\pi$). Let $\mathfrak{M}_{Th,\mathfrak{p}}$ be the moduli space of Euclidean cone metrics on $S^2$ with cone points of apex curvature $\frac{4\pi n_j}{N}$ (i.e. cone angle $2\pi\left(1-\frac{2n_j}{N}\right)$) and of total area $1$. 

Let $\mathfrak{M}_{EB,\mathfrak{p}}\left(L,\tau; V\right)$ be the space of Einstein-Bogomol'nyi metrics with volume $V$ whose Higgs field is compatible with the partition $\mathfrak{p}$. What we've shown about the large volume limit can be roughly phrased as follows: \emph{``for stable $\mathfrak{p}$,  $
\mathfrak{M}_{Th, \mathfrak{p}} \text{ is an adiabtic limit of }\mathfrak{M}_{EB,\mathfrak{p}}\left(L, \tau; V\right)/PSL(2,\mathbf{C})$ as $V\to +\infty$
in certain sense''}. The  phenomenon  of ``colliding'' cones points (which is responsible for the incompleteness of each $\mathfrak{M}_{Th,\mathfrak{p}}$) should have similar companion phenomenon of ``merging'' of  vortices in $\mathfrak{M}_{EB, \mathfrak{p}}(L, \tau; V)$, at least for $V$ large enough. 

To understand better the structure of $\mathfrak{M}_{EB}\left(L,\tau\right)$ or more generally the moduli space of solutions to K\"ahler-Yang-Mills equations \cite{AlGaGa2,Al-Ga-Ga-P}, it is inevitable to discuss the ``degeneration'' of the metric when the Higgs field varies, which represents a variation of complex structure \cite{FT} on the $SU(2)$ bundle over $\mathbf{P}^1\times \mathbf{P}^1$.  In this perspective, we should mention a general existence result about ``multiple strings'' Einstein-Bogomol'nyi metrics on $\mathbf{C}$ \cite[Theorem 10.4.1]{Yang4} since it might be related to the ``degeneration'' problem in the full moduli space $\mathfrak{M}_{EB}\left(L, \tau\right)$ when we allow $\bm\phi$ to vary along the sequence and volume goes to $+\infty$ simultaneously. In terms of the equation \eqref{eqn:GV0}, this existence result can be phrased as follows: for any holomorphic function $\bm\phi$ on $\mathbf{C}$ (which is a holomorphic section of the trivial line bundle on $\mathbf{C}$) vanishing at $p_1, \cdots, p_d$ with multiplicities $n_1, \cdots, n_d$ respectively,  finite energy solutions exist under the condition on the total string number 
\[
N':= n_1+\cdots + n_d<\frac{1}{\alpha\tau}.
\]
It is also showed that the underlying Riemannian metrics on $\mathbf{C}$ of those solutions are complete if and only if 
\[
N':=n_1+\cdots + n_d\leq \frac{1}{2\alpha\tau}.
\]
It follows from the \emph{quasi-isometry} relation \cite[Equation (113)]{Yang} of the obtained metric with standard model metric on complement of a compact set in $\mathbf{C}$ that the metric is asymptotically conical in case $N'<\frac{1}{2\alpha\tau}$, and asymptotically cylindrical in case $N'=\frac{1}{2\alpha \tau}$. Taking into consideration of all the results obtained in this article, we can make the following reasonable conjecture:

\begin{conjecture}
	Let $\left(\omega_i, h_i, \bm\phi_i\right)\in \mathfrak{M}_{EB}\left(L,\tau\right)$ be a sequence of Einstein-Bogomol'nyi metrics with $\bm\phi_i$ all being polystable, and $\bm\phi_i\to \bm\phi$ is also polystable.
	\begin{enumerate}
		\item If $\Vol_{\omega_i}\to V\in \left( \underline V, +\infty\right)$, then there exists some sequence $\sigma_i\in PSL(2,\mathbf{C})$ such that $\sigma_i^*\left( \omega_i, h_i, \bm\phi_i\right)$ converges to some $\left(\omega, h, \bm\phi\right)$;
		\item If $\Vol_{\omega_i} \to +\infty$, then any nonflat Cheeger-Gromov limit of the sequence can be realized as one of the ``multiple strings" solutions of Yang.
	\end{enumerate}
\end{conjecture}

\begin{acknowledgements}
	This study was funded by the start up grant supported by ShanghaiTech University under the No. 2018F0303-000-03. The author would like to thank Professor Mingliang Cai and Oscar Garcia-Prada for some valuable discussions. He would also like to express the gratitude to Mario Garcia-Fernandez, Vamsi Pingali and Song Sun for their interests in the work. 
\end{acknowledgements}


\begin{thebibliography}{99}
	\bibitem{AlGaGa}
	L. \'Alvarez-C\'onsul, M. Garcia-Fernandez and O. Garc\'ia-Prada,
	\emph{Coupled equations for K\"ahler metrics and Yang--Mills connections}, Geom. Top. \textbf{17} (2013) 2731--2812.
	
	\bibitem{AlGaGa2}
	\bysame,
	\emph{Gravitating vortices, cosmic strings, and the K\"ahler--Yang--Mills equations}, Comm. Math. Phys. \textbf{351} (2017) 361--385.
	
	\bibitem{Al-Ga-Ga-P} 
	L. \'Alvarez-C\'onsul, M. Garc\'ia-Fern\'andez, O. Garc\'ia-Prada, and V. Pingali,
	\emph{gravitating vortex and the Einstein-Bogomol'nyi equations}, Math. Ann. 379 (2021), no. 3-4, 1651-1684.
	
	\bibitem{Al-Ga-Ga-P-Y}
		L. \'Alvarez-C\'onsul, M. Garcia-Fernandez and O. Garc\'ia-Prada, V. P. Pingali, C.-J. Yao, \emph{
	Obstructions to the existence of solutions of the self-dual Einstein-Maxwell-Higgs equations on a compact surface}, Bull. Sci. Math. 183 (2023), Paper No. 103233, 14 pp.
	
	
	\bibitem{BM}
	J.M. Baptista, N. S. Manton,
	\emph{The dynamics of vortices on $S^2$ near the Bradlow limit}, Journal of Mathematical Physics, Volume 44, No. 8, 2003. 

    \bibitem{BaMc}
    R. Bartnik, J. McKinnon,\emph{
Particlelike solutions of the Einstein-Yang-Mills equations}.
Phys. Rev. Lett.61(1988), no.2, 141-144.
	
	
	
	\bibitem{CHMcY}
	 X.-F. Chen,S. Hastings, J. R. McLeod, Y.-S. Yang, \emph{A nonlinear elliptic equation arising from gauge field theory and cosmology}. Proc. Roy. Soc. London Ser. A 446 (1994), no. 1928, 453-478.
	
	\bibitem{Chow}
	B. Chow, D. Knopf, \emph{
	The Ricci flow: an introduction. }
	Mathematical Surveys and Monographs, 110. American Mathematical Society, Providence, RI, 2004. xii+325 pp.

   \bibitem{ComGib}
   A. Comtet, G. W. Gibbons, \emph{Bogomol'nyi bounds for comic strings}. Nucl. Phys. B 299, 719-733 (1988).  
	
	\bibitem{Doan}
	A. Doan, 
	\emph{Adiabatic limits and Kazdan-Warner equations}, Calculus of Variations and PDEs, 2018.
	
	
	\bibitem{Sohn}
    J. Sohn, 
     \emph{Mountain pass solution for the self-dual Einstein-Maxwell-Higgs model on compact surfaces}, J. Math. Phys. 64, 071505(2023);
     
	
	\bibitem{HS}
	J. Han, J. Sohn,
	\emph{On the self-dual Einstein-Maxwell-Higgs equation on compact surfaces}, Discrete and Continuous Dynamical Systems, Volume 39, Number 2, 2019, 819-839. 
	

	
	\bibitem{FPY}
	M. Garcia-Fernandez, V. P. Pingali, and C.-J. Yao, 
	\emph{Gravitating vortices with positive curvature}, Adv. Math. 388 (2021), Paper No. 107851. 
	
	\bibitem{FT}
	M. Garcia-Fernandez, C. Tipler, \emph{Deformation of complex structures and the coupled K\"ahler-Yang-Mills equations}. J. Lond. Math. Soc. (2) 89 (2014), no. 3, 779-796. 
	
	
	\bibitem{HJS}
	M.-C. Hong, J. Jost, M. Struwe,
	\emph{Asymptotic limits of a Ginzburg-Landau type functional}. (English summary) Geometric analysis and the calculus of variations, 99-123, Int. Press, Cambridge, MA, 1996. 
	
	\bibitem{Kl}
	W. Klingenberg, \emph{Contributions to Riemannian geometry in the large}. Ann. Math. 69(3), 654-666 (1959).
	
	\bibitem{KN}
	V. Koziarz, D.-M. Nguyen, \emph{Complex hyperbolic volume and intersection of boundary divisors in moduli spaces of pointed genus zero curves}. Ann. Sci. \'Ec. Norm. Sup\'er. (4) 51 (2018), no. 6, 1549-1597.
	
	\bibitem{Linet}
	B. Linet, \emph{On the supermassive \text{U}(1) gauge cosmic strings}. Class. Quantum Grav. 7, L75-L79 (1990).
	
	
	\bibitem{Mum}
	D. Mumford, J. Fogarty, F. Kirwan, \emph{Geometric Invariant Theory}, 3d edition, Springer-Verlag, 1994.
	
	\bibitem{Sy}
	S. Serfaty,
	\emph{Vortices in the Ginzburg-Landau model of superconductivity}. International Congress of Mathematicians. Vol. III, 267-290, Eur. Math. Soc., Z\"urich, 2006.
	

	\bibitem{SWYM}J. Smoller, A. Wasserman, S.-T. Yau, B. McLeod, \emph{Smooth static solution of the Einstein/Yang-Mills equations}, Communications in Mathematical Physics, 143(1991), no. 1, 115-147. 
	
	\bibitem{Th}
	W. Thurston, \emph{Shapes of polyhedra and triangulations of the sphere}. The Epstein birthday schrift, 511-549, 
	Geom. Topol. Monogr., 1, Geom. Topol. Publ., Coventry, 1998. 
	
	\bibitem{Yang}
	Y. Yang,
	\emph{Prescribing Topological Defects for the Coupled Einstein and Abelian Higgs Equations}. Communications in Mathematical Physics. 170, 541-582(1995). 
	
	\bibitem{Yang4}
		\bysame,
    \emph{Solitons in Field Theory and Nonlinear Analysis}. Springer Monographs in Mathematics, Springer, New York (2001)

	
	\bibitem{Yang5}
	\bysame,
	\emph{Static cosmic strings on $S^2$ and criticality},
	Proc. Roy. Soc. Lond. A {\bf 453} (1997) 581--591.
	
		
\end{thebibliography}
\end{document}